\documentclass[11pt]{article}
\usepackage[a4paper, total={6in, 8in}]{geometry}

\usepackage{url}
\usepackage{graphicx,wrapfig,lipsum}
\usepackage[charter]{mathdesign}
\usepackage[T1]{fontenc}

\usepackage{caption}
\usepackage{amsmath,amsthm,amsfonts}
\usepackage{faktor}
\usepackage{tikz}
\usepackage{enumitem}
\usepackage{comment}
\usepackage{lipsum}
\usepackage{authblk}
\makeatletter
\title{The intersections of the lower central series and the subgroups of finite $p$-index of Generalized Baumslag-Solitar tree groups}
\author{V. Metaftsis} 
\author{D. Tsipa\thanks{
	The research work was supported by the Hellenic Foundation for
			Research Innovation (HFRI) under the 3rd Call for HFRI PhD
			Fellowships.  (Fellowship Number: 5161) \\  \protect\includegraphics[width=0.17\textwidth]{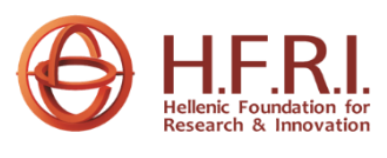} 	}} 
\affil[ ]{Department of Mathematics, University of the Aegean, Karlovassi, 832 00 Samos, Greece}
\affil[ ]{\textit {vmet@aegean.gr, dtsipa@aegean.gr}}
\date{}
\usepackage[overload]{empheq}
\theoremstyle{theorem} 
\newtheorem{theorem}{Theorem}[section]
\newtheorem{corollary}[theorem]{Corollary}
\newtheorem{lemma}[theorem]{Lemma}
\newtheorem*{lemma*}{Lemma}
\newtheorem*{theorem*}{Theorem}
\newtheorem{proposition}[theorem]{Proposition}
\newtheorem*{proposition*}{Proposition}
\theoremstyle{definition}

\newtheorem*{example*}{Example}    

\newtheorem*{remark*}{Remark}

\newcommand{\Z}{\mathbb{Z}}

\newcommand{\om}{\omega}

\def\g{\gamma}
\def\d{\delta}
\def\a{\alpha}
\def\b{\beta}
\def\g{\gamma}
\def\e{\epsilon}

\def\f{\phi}
\def\om{\omega}
\def\k{\kappa}
\def\l{\lambda}
\def\<{\langle}
\def\>{\rangle}
\newcommand\restrict[1]{\raisebox{-.5ex}{$|$}_{#1}}
\newcommand{\bigslant}[2]{{\raisebox{.2em}{$#1$}\left/\raisebox{-.2em}{$#2$}\right.}}

\begin{document}
	
	\maketitle
	\begin{abstract}
	For a Generalized Baumslag-Solitar group $G$ with underling graph a tree, we calculate the intersection $\gamma_{\om}(G)$ of the lower central series and the intersection $(N_{p})_{\om}(G)$ of the subgroups of finite index some power of a prime $p$.
	\end{abstract}
\section{Introduction}
Generalized Baumslag-Solitar groups (GBS-groups) are fundamental groups of finite graphs of groups with infinite cyclic vertex and edge groups. These groups have interesting group-theoretic properties and have been the subject of recent research (see for example \cite{Delga, Fore,LevittA, Robi}).

In the case where the underlying graph is a simple loop we have the Baumslag-Solitar groups which are the $1$-relator groups
$$BS(m,n)=\langle t, a \hspace{1mm}|\hspace{1mm}ta^{m}t^{-1}=a^{n}\rangle.$$
where $m, n$ are non-zero integers. 

 Baumslag-Solitar groups seem to have first appeared in the literature in \cite{BS} as they were defined by Gilbert Baumslag and Donald Solitar in order to provide examples of non-Hopﬁan groups. These groups have played a really central r\^ole in combinatorial and geometric group theory. In several situations they have provided examples which mark boundaries between different classes of groups and also they often provide a test for various theories.
 
 	 Let $G$ be a group. We say that $G$  is residually a $\mathcal{P}$ group if for every element $1 \neq g \in G$ there is a homomorphism $f$ from $G$ to a group with the  property $\mathcal{P}$ such that $f(g)$ is not trivial. 
 	 
 	  For elements $a,b$ of $G$, we write $[a,b]$ for the commutator of $a$ and $b$, that is $[a,b]=a^{-1}b^{-1}ab$. For subgroups $A$ and $B$ of $G$ we write $[A,B]=\langle [a,b]\hspace{1mm}:\hspace{1mm}a\in A, b\in B\rangle$. For a positive integer $i$, let $\g_{i}(G)$ be the $i-$th term of the lower central series of $G$. We denote by $\gamma_{\omega}(G)$ the intersection of all terms of the lower central series of $G$, that is $\gamma_{\omega}(G)=\bigcap\limits_{i\geq 1} \g_{i}(G)$. A group $G$ is residually nilpotent if and only if $\gamma_{\omega}(G)$ is trivial.

 	 We denote the intersection of all finite index normal subgroups of $G$ with index some power of a prime number $p$ by $(N_{p})_\omega(G)$. A group is residually a finite $p$-group if and only if $(N_{p})_\omega(G)$ is trivial. 
 	 
 	 In his semial work in \cite{Gruen},  Gruenberg proved various results concerning the residual properties of groups. Also, in \cite{Mold} (see also \cite{Mold-ru}) Moldavanskii calculated the intersection $(N_{p})_\omega(G)$ for a Baumslag-Solitar group $G$ while in \cite{KMP} Kofinas, Metaftsis and Papistas calculated the intersection $\gamma_{\om}(G)$. Also, Sokolov, in \cite{sok} provides criteria for a GBS group to be residually nilpotent,  residually torsion-free nilpotent and residually free..
 
 In the present work we calculate the intersection $\gamma_{\om}(G)$ and the intersection $(N_{p})_{\om}(G)$ for a GBS tree group,  that is a GBS group with underlying graph is a tree. In fact, in our main result Theorem \ref{amalgamatedprod} we prove that $\g_{\om}(G)$ is generated by $\g_{\om}(H_{ij})$ for every $H_{ij}$ a subgroup of $G$ generated by any two vertex groups $\<x_i,x_j\>$. We prove a similar result also for $(N_p)_{\om}(G)$ (Theorem \ref{Npw_tree}).
 
 Although the final generating sets are rather complicated, our results show that both $\g_{\om}(G)$ and $(N_p)_{\om}(G)$ are generated in a rather natural way.  Finally,  at the end of the paper we provide an appendix with a specific example for which we demonstrate the generating sets for both $\gamma_{\om}(G)$ and $(N_3)_{\om}(G)$.

	\section{Definitions and auxiliary results}

	Let $G$ be a group. A $G$-tree is a tree with a $G$-action by automorphisms, without inversions. A $G$-tree is proper if every edge stabilizer is a subgroup of the vertex stabilizers for the vertices that correspond to this edge. It is minimal if there is no proper $G$-invariant subtree and cocompact if the quotient graph is finite. The quotient of the action is the graph of groups with vertex groups the vertex stabilizers and edge groups the edge stabilizers.

	Given a $G$-tree $X$, an element $g\in G$ is called elliptic if it stabilizes a vertex of $X$ and hyperbolic otherwise. If $g$ is hyperbolic then there is a unique $G$-invariant line in $X$, the axis of $g$, on which $g$ acts as a translation. A subgroup $H$ of $G$ is elliptic if it fixes a vertex.
	
	A generalized Baumslag-Solitar tree is a $G$-tree whose vertex stabilizers and edge stabilizers are all infinite cyclic. Such a group $G$ is called generalized Baumslag-Solitar group (GBS group). The quotient graph of groups has all vertex and edge groups isomorphic to $\Z$ and the inclusion maps are multiplication by various non-zero integers.
	
	The group $G$ is torsion free since any torsion group intersects each conjugate of a vertex group trivially which implies that is free and therefore trivial.  Moreover, all elliptic elements are commensurable, i.e. they have a common power.
	
	GBS-groups are presented by labeled graphs, that is finite graphs where each oriented edge $e$ has label $\lambda_e\neq 0$. If $\Gamma$ is a finite graph with vertex and edge groups isomorphic to $\Z$ then we have a graph of groups with fundamental group a GBS group.  A presentation for a GBS group $G$ can be obtained as follows. Choose a maximal tree $T\subset \Gamma$. There is one generator $g_v$ for every vertex of $\Gamma$ and one relator $t_e$ for every edge in $\Gamma\setminus T$ (non-oriented). Each non-oriented edge contributes  a relation in the presentation. The relations are
	$$(g_{\alpha(e)})^{\lambda_e}=(g_{\alpha(\overline{e})})^{\lambda_{\overline{e}}}\ \ \mbox{if}\ \ \  e\in T$$
	$$t_e(g_{\alpha(e)})^{\lambda_e}t_e^{-1}=(g_{\alpha(\overline{e})})^{\lambda_{\overline{e}}}\ \   \mbox{if}\ \ e\in\Gamma\setminus T$$
	The above is the standard presentation of $G$ and the generating set is the standard generating set associated to $\Gamma$ and $T$.
	
	Notice that the group $G$ represented by $\Gamma$ does not change if we change the sign of all labels near a vertex $v$ or the labels carried by a non-oriented edge. These are called admissible sign changes. Notice also that there may be infinitely many labeled graphs representing the same $G$. If $\Gamma$ is a loop, then $G$ is the known Baumslag-Solitar group.
	In case $\Gamma$ is a tree we will call $G$ a generalized Baumslag-Solitar tree group.

	We will now present some auxiliary results that will help us calculate the intersections $\gamma_{\omega}(G)$ and $(N_{p})_{\om}(G)$ in the case where the underlying graph of the GBS is a segment on which is the basis for calculating the general case of the GBS tree group. 

	Let $Z(G)$ be the center of a group $G$.
	\begin{lemma}\label{intersectionGBS}
	Let $G$ be a GBS group, where the underlying graph is a tree $X$. Then $G'\cap Z(G)$ is trivial.
	\end{lemma}
\begin{proof}
Choose an orientation for $X$. Then the group $G$ has the following presentation. 
	$$G=\< x_{v_{i}} \hspace{1mm}|\hspace{1mm}  x_{\a(e_{j})}^{\l_{e_{j}}}= x_{\om(e_{j})}^{\bar{\l}_{{e}_{j}}} \>, \hspace{1mm} v_{i}\in V(X), e_{j}\in E(X)$$
	Raising the relations of $G$ to appropriate powers we can find integers $r_{v_{i}}$ such that $x_{v_{i}}^{r_{v_{i}}}=x_{v_{i'}}^{r_{v_{i'}}}$ for all $v_{i},v_{i'}\in V(X)$ . Then $Z(G)=\< x_{v_{i}}^{r_{v_{i}}}\>$,
		that is, the center of $G$ is a subgroup of each vertex group $G_{v_{i}}=\< x_{v_{i}}\>$. 
		  Notice that in the case where $v_{i}$ and $v_{i'}$ are adjacent vertices  incident with an edge $e$ then $r_{v_{i}}=\l_{e}N_{e}$ and $r_{v_{i'}}=\bar{\l}_{{e}}N_{e}$ for the same integer $N_{e}$.
 
Let $r=lcm(r_{v_{i}})$, $v_{i}\in V(X)$, and $H$ be the infinite cyclic group $\<y\>\cong \mathbb{Z}$. We define the map $\f: G \to H$ with $\f(x_{v_{i}})= y^{\frac{r}{r_{v_{i}}}}$, for all $v_{i}\in V(X)$. The map $\f$ is a homomorphism since it preserves the relations of $G$ (it follows from the previous observation) and every vertex group embeds in to $\<y\>$.  Since $G'\subseteq Ker\f$,  the result follows. 
\end{proof}
\begin{lemma}\label{intersectionfinite}
	Let $G$ be the fundamental group of a graph of groups $(\mathcal{G},X)$ where $X$ is a tree with vertex groups finite cyclic $p$-groups. Then $G'\cap Z(G)$ is trivial.
\end{lemma}
\begin{proof}
	Assume that $Z(G)$ is not trivial (otherwise the result is obvious). Since  $G$ is the fundamental group of a tree of groups where the vertex groups are finite cyclic $p$-groups then the center $Z(G)$ is a non-trivial cyclic subgroup $\langle x_{v_{i}}^{r_{i}} \rangle$ of each edge group. If $G_{v_{i}}= \langle x_{v_i} \rangle$ we have that $Z(G)=\<x_{v_{i}}^{m_{i}}\>$ for all $v_{i}\in V(X)$.
	
	The group $G$ has the following presentation. 
	$$G=\< x_{v_{i}} \hspace{1mm}|\hspace{1mm} x_{v_{i}}^{p^{k_{i}}}=1, x_{\a(e_{j})}^{\l_{e_{j}}}= x_{\om(e_{j})}^{\bar{\l}_{{e}_{j}}} \>, \hspace{1mm} v_{i}\in V(X), e_{j}\in E(X)$$
	
	Let $k=max\{k_{i}\}$ and thus for some $v_{r}\in V(X)$ we have  $ord(x_{v_{r}})=p^{k}$. Let $H$ be the group $H=\< y\>$ with $ord(y)=p^{k}$, and thus $\< y\>\cong \mathbb{Z}_{p^{k}}$. Therefore $\<x_{v_{r}}\>\cong \< y\>$. Each vertex group $\<  x_{v_{i}} \>$ can be embedded to the group $H$ and it's image will be a subgroup of order $d_{i}$, where $d_{i}$ is a divisor of $p^{k}$, i.e. $d_{i}\in \{p^{\a}\}_{\a=1}^{r}$. Therefore we can define the map $\f : G \to H $ such that $\f(x_{v_{i}})=y^{\frac{p^{k}}{p^{k_{i}}}}$, for all $v_{i}\in V(X)$. The map $\f$ is a homomorphism, since it preserves all relations. Indeed $\f(x_{v_{i}}^{p^{k_{i}}})=y^{p^{k}}=1$, for all $v_{i}\in V(X)$. Moreover, notice that the relation $x_{\a(e_{j})}^{\l_{e_{j}}}= x_{\om(e_{j})}^{\bar{\l}_{{e}_{j}}}$ is preserved since $ord(x_{\a(e_{j})}^{\l_{e_{j}}})=ord(x_{\om(e_{j})}^{\bar{\l}_{{e}_{j}}})$ and due to the fact that there is a one-to-one correspondence between the subgroups of a finite cyclic group and their order. 
	
	Now assume that there is a non trivial element $z\in Z(G)$ such that $z\in G'$. We have that $z=(x_{v_{r}})^{s}$, for some $s\in \mathbb{Z}$. We have that $G'\subseteq Ker\f$ and thus $z\in Ker\f$. Let $\f\restrict{{G_{v_{r}}}}$ be the restriction of $\f$ to the group  $G_{v_{r}}=\< x_{v_{r}} \>$. Then $z\in Ker\f\restrict{{G_{v_{r}}}}$, which is a contradiction, since the map $\f\restrict{{G_{v_{r}}}}$ that maps $x_{v_{r}}$ to $y$ is an isomorphism. 
\end{proof}
	\begin{lemma}\label{Lemma0}
	Let $G$ be a group and $N$ be a normal subgroup of $G$. Then $\gamma_{i}(G/N)=\g_{i}(G)N/N$, for every positive integer $i$.
\end{lemma}
\begin{proof}
	We will prove this Lemma using induction on $i$.
	
	For $i=1$ we have that $\gamma_{1}(G/N)=G/N=\g_{1}(G)N/N$. Now assume that for some $i>1$ we have that $\gamma_{i}(G/N)=\g_{i}(G)N/N$.
	
	Then $\g_{i+1}(G/N)=[\g_{i}(G/N), G/N]$ which by the inductive hypothesis is equal to \\ $[\g_{i}(G)N/N,G/N]$ $=$ $\langle [aN,bN]\hspace{1mm} :\hspace{1mm} a\in \g_{i}(G), b\in G \rangle$. Let $[aN,bN]\in\g_{i+1}(G/N)$, $a\in \g_{i}(G),$ $ b\in G$. Since  $[aN,bN]=[a,b]N\in \g_{i+1}(G)N/N$ we have that $\gamma_{i+1}(G/N)$ $\subseteq \g_{i+1}(G)N/N$.
	
	For the converse we have $\g_{i+1}(G)N/N=[\g_{i}(G),G]N/N$ $=$ $\langle [a,b] \hspace{1mm}: \hspace{1mm} a\in \g_{i}(G), b\in G\rangle N/N$. Let $[a,b]N\in \g_{i+1}(G)N/N$. Thus $[a,b]N=[aN,bN]$, where $a\in \g_{i}(G), b\in G$. Therefore $[a,b]N\in \gamma_{i+1}(G/N)$ and consequently we have the equality $\gamma_{i+1}(G/N)=\g_{i+1}(G)N/N$.
\end{proof}
	\begin{lemma}\label{Lemma1}
		Let $G$ be the fundamental group of a graph of groups where the vertex groups are all either infinite cyclic groups (and thus $G$ is a GBS group) or finite cyclic $p$-groups. Then $\gamma_{\omega}(G)\cong\gamma_{\omega}(G/Z(G))$.
	\end{lemma}
\begin{proof}
	Since $Z(G)$ is a normal subgroup of $G$ by Lemma \ref{Lemma0} we know that $\gamma_{i}(G/Z(G))=\g_{i}(G)Z(G)/Z(G)$ for every positive integer $i$.

	The group $\gamma_{i}(G/Z(G))=\g_{i}(G)Z(G)/Z(G)$ from the second isomorphism theorem for groups is isomorphic to $\g_{i}(G)/(\g_{i}(G)\cap Z(G))$.
	
	Now for each $i>1$, since the center of $G$ is contained in every vertex group, we have that $\g_{i}(G)\cap Z(G)\subseteq \g_{i}(G)\cap G_{v}\subseteq \g_{2}(G)\cap G_{v}$, where $G_{v}$ is the abelian vertex group of $G$. But from Lemma  \ref{intersectionGBS} and Lemma \ref{intersectionfinite} we have that $\g_{2}(G)\cap G_{v}$ is trivial and thus $\g_{i}(G)\cap Z(G)$ is trivial. 
	
	The result now follows from the definition of $\gamma_{\omega}(G)$.
\end{proof}
Let $p$ be a prime number. Using the above Lemma, we will prove that the fundamental group of a tree of finite cyclic $p$-groups is residually nilpotent and hence residually finite-$p$. The equivalence of these two properties follows from the following Proposition.
\begin{proposition}[see \cite{Varsos}]\label{Varsos}
	Let ($\mathcal{G},\Gamma$) be a graph of groups, where $\Gamma$ is a connected finite graph with vertex groups $G_{v}$, $v\in V(X)$, finite $p$-groups. The fundamental group $G=\pi_1(\mathcal{G},\Gamma)$ is residually finite-$p$ if and only if it is residually nilpotent. 
\end{proposition}
The residual nilpotency of the fundamental group of finite cyclic $p$-groups can easily be proved in the case where the underlying graph is a segment.
\begin{proposition}\label{segment}
Let $G$ be the fundamental group of a segment of finite cyclic $p$-groups, that is the free product of two finite cyclic $p$-groups $\langle x\rangle =\mathbb{Z}_{p^{k}}$ and $\langle y\rangle =\mathbb{Z}_{p^{l}}$   amalgamated over the subgroups $\langle x^{r}\rangle=\langle y^{s}\rangle$, for some $r,s\in \mathbb{Z}$. Then $G$ is residually nilpotent and thus residually finite-$p$.
\end{proposition}
\begin{proof}
By Lemma \ref{Lemma1} we have that $G$ is residually nilpotent if and only if $G/Z(G)$ is residually nilpotent. But since the center $Z(G)$ of $G$ is the amalgamated subgroup $\langle x^{r}\rangle=\langle y^{s}\rangle$ we have that $G/Z(G)\cong \Z_{r} * \Z_{s} $. 

The groups $\Z_{r}$ and $ \Z_{s}$ are also finite cyclic $p$-groups and therefore residually finite-$p$. Therefore, using Gruenberg's result we have that their free product $G/Z(G)$ is also residually finite-$p$  and thus residually nilpotent. Consequently, $G$ is residually nilpotent and thus residually finite-$p$ by Proposition \ref{Varsos}.
\end{proof}
We mention here that the above result can also be obtained by Theorem 1 in \cite{RozSok}.

\vspace{2mm}
We will now generalize the above result in the case where the underlying graph is a tree $T$. In order to prove this, we need the following Lemma.
\begin{lemma}\label{center}
		Let $(\mathcal{G},T)$ be a graph of groups, with $T$ be a finite tree and $G_{v_{i}}$ be the vertex groups which are finite cyclic $p$-groups. Then $Z(G)=G_{e}$, for some $e\in E(T)$.
\end{lemma}
\begin{proof}
		We will prove the Proposition using induction on the number of the vertices of $T$.
		Notice that in the case where $V(T)=2$ the statement is obvious.
		
		Let $V(T)=3$, i.e. $G$ is a group with presentation
		$$G=\< x,y,z \hspace{1mm}|\hspace{1mm}x^{p^{k_{1}}}=1,y^{p^{k_{2}}}=1, z^{p^{k_{3}}}=1, x^{\k}=y^{\l}, y^{\mu}=z^{\nu}\>$$
		
	where $k_{1},k_{2},k_{2},\k,\l,\mu,\nu \in \mathbb{N}$.
	
	We have that $G_{e_{1}}=\< x^{\k} \> =\< y^{\l} \>$ and  $G_{e_{2}}=\< y^{\mu} \> =\< z^{\nu} \>$. Since  $\< y^{\l} \>$ and $\< y^{\mu} \>$ are both subgroups of $\< y\> \cong \mathbb{Z}_{p^{k_{2}}}$ and  since $p$ is a prime number, we have that $\< y^{\l} \> \leqslant\< y^{\mu} \>$ or $\< y^{\mu} \> \leqslant\< y^{\l} \>$. Without loss of generality we may assume that $\< y^{\l} \> \leqslant\< y^{\mu} \>$ and therefore $ y^{\l}\in \< y^{\mu} \> \Rightarrow y^{\l}=(y^{\mu})^{\rho}$, for some $\rho\in \mathbb{Z}$. Then we have that \vspace*{-2mm}
	\begin{itemize}
		\item[]$xy^{\l}x^{-1}=xx^{\k}x^{-1}=x^{\k}=y^{\l}$ \vspace*{-2mm}
		\item[]$yy^{\l}y^{-1}=y^{\l}$  \vspace*{-2mm}
		\item[]$zy^{\l}z^{-1}=z(y^{\mu})^{\rho}z^{-1}=z(z^{\nu})^{\rho}z^{-1}=(z^{\nu})^{\rho}=(y^{\mu})^{\rho}=y^{\l}$  \vspace*{-2mm} 
	\end{itemize}
Hence $Z(G)=\<y^{\l}\>=G_{e_{1}}$. Similarly, in the case where  $\< y^{\mu} \> \leqslant\< y^{\l} \>$ we have that  $Z(G)=\<y^{\mu}\>=G_{e_{2}}$.

Assume now the statement of the Lemma holds in the case where $V(T)=n$.  We will prove that the statement holds if $T$ has $n+1$ vertices.

Let $G$ be the group with presentation
$$G=\< x_{v_{i}} \hspace{1mm}|\hspace{1mm} x_{v_{i}}^{p^{k_{i}}}=1, x_{\a(e_{j})}^{\l_{e_{j}}}= x_{\om(e_{j})}^{\l_{\bar{e}_{j}}} \>, \hspace{1mm} v_{i}\in V(T), e_{j}\in E(T)$$

The edge groups are the groups $G_{e_{j}}=\<x_{\a(e_{j})}^{\l_{e_{j}}}\>=\<x_{\om(e_{j})}^{\l_{\bar{e}_{j}}}\>$ for all $e_{j}\in E(T)$.

Now let $T'$ be the subtree of $T$ that occurs after removing an external vertex and let $H$ be the fundamental graph of groups that corresponds to $T'$. 
$$H=\< x_{v_{i}} \hspace{1mm}|\hspace{1mm} x_{v_{i}}^{p^{k_{i}}}=1, x_{\a(e_{j})}^{\l_{e_{j}}}= x_{\om(e_{j})}^{\l_{\bar{e}_{j}}} \>, \hspace{1mm} v_{i}\in V(T'), e_{j}\in E(T')$$
Notice that $H$ is a subgroup of $G$. By the inductive hypothesis we have that there exists $s$ such that $Z(H)=G_{e_{s}}$.

We have that $G_{e_{s}}= \< x_{v_{i}}^{\mu_{i}} \>$ for all $v_{i}\in V(T')$ where if $v_{i}=\a(e_{s})$ then $\mu_{i}=\l_{e_{s}}$ while if $v_{i}=\om(e_{s})$ then $\mu_{i}=\l_{\bar{e}_{s}}$. 

Let $\hat{e}$ be the edge such that $\hat{e}\in T\setminus T'$ and assume that $\a(\hat{e})=v_{n}$. We have that $G_{\hat{e}}= \< x_{v_{n}}^{\l_{\hat{e}}} \>$ which is a subgroup of $\< x_{v_{n}} \>\cong \mathbb{Z}_{p^{k_{n}}}$. Moreover,  $G_{e_{s}}= \< x_{v_n}^{\mu_{n}} \>$ which is a subgroup of $\< x_{v_n} \>\cong \mathbb{Z}_{p^{k_{n}}}$. Therefore, since $p$ is a prime number, either  $\< x_{v_n}^{\l_{\hat{e}}} \> \leqslant \< x_{v_n}^{\mu_{n}} \>$ or $\< x_{v_n}^{\mu_{n}} \> \leqslant \< x_{v_n}^{\l_{\hat{e}}} \>$. Hence we have that either $Z(G)=G_{e_{s}}$ or  $Z(G)=G_{\hat{e}}$ and thus the statement holds. 
\end{proof}
\begin{proposition}\label{ptree}
	Let $G$ be the fundamental group of a tree $T$ of finite cyclic $p$-groups. Then $G$ is residually nilpotent and thus residually finite-$p$.
\end{proposition}
\begin{proof}
		We will prove the Proposition using induction on the number of the vertices of $T$.
		
		Let $V(T)=2$. Then by Proposition \ref{segment} the result holds.
		Assume now that the statement of the Proposition holds for a graph of groups with underlying tree $T$ with $|V(T)|=n$. We will prove that the statement holds if the tree has $n+1$ vertices.
		
		By Lemma \ref{center} we have that $Z(G)$ is equal to an edge group for an edge $e\in E(T)$  (that is one of the amalgamating subgroups) and therefore the group $G/Z(G)$ is isomorphic to a free product of $G_{1}$, $G_{2},\dots,G_{m}$ which are fundamental groups of $m$ subtrees $T_{1}$, $T_{2},\dots,T_{m}$ of $T$ and therefore each has at most $n$ vertices. Hence, using the inductive hypothesis $G_{1}$, $G_{2},\dots,G_{m}$ are residually nilpotent and thus by Proposition \ref{Varsos}  residually finite-$p$. Therefore, using Gruenberg's result we have that their free product $G/Z(G)$ is also residually finite-$p$  and thus residually nilpotent. Using Lemma \ref{Lemma1} the result follows.
\end{proof}
The above result can also be verified using the criterion for residual $p$-finiteness of arbitrary graphs of finite $p$-groups by G. Wilkes (see \cite{Wilkes}).
	\begin{lemma}\label{Lemma2}
	Let $G$ be a group and $a,b\in G$ such that  $a^{k}= b^{l}$, for some $k,l\in \mathbb{Z}$. If $\k,\l$ are divisors of $k$ and $l$ respectively with $k=\k s$,  $l=\l r$  and $gcd(\k,\l)=1$ then $[a^{s},b^{r}]\in \gamma_{\omega}(G)$.
\end{lemma}	
\begin{proof}
	Let $a^{s}=x$ and $b^{r}=y$. We want to prove that the commutator $g=[x,y]$ belongs to $\gamma_{i}(G)$ for each $i$.
	Assume on the contrary that $1\neq g\notin \gamma_{\omega}(G)$. Therefore there exists a  homomorphism $f: G \to N$, where $N$ is a nilpotent group, such that $f(g)\neq 1$.
	
	 Now since $N$ is a finitely generated nilpotent group we have that $N$ is polycyclic and thus residually finite. Therefore there exists a homomorphism $\f: G \to \tilde{N}$ such that $\f(g)=[\f(x),\f(y)]$ $\neq 1$, where $\tilde{N}$ is a finite nilpotent group. Since $\tilde{N}$ is a finite nilpotent group it can be written as the direct product of its Sylow subgroups. Moreover, since $[\f(x),\f(y)]\neq 1$ we have that $\f(x)$ and $\f(y)$ belong to the same Sylow $p$-group $P$ of $\tilde{N}$.
	
	Let $H$, $K$ be the subgroups of $P$ that are generated by $\f(x)$ and $\f(y)$ respectively. Thus $H=\langle \f(x)\rangle\cong \mathbb{Z}_{p^\nu}$ and $K=\langle \f(y)\rangle\cong \mathbb{Z}_{p^\mu}$. Let's assume that $ord(\f(x))=ord(\f(x)^{\k})$. Since $x^{\k}=y^{\l}$ we have that $\f(x)^{\k}=\f(y)^{\l}$ and thus $\langle \f(x)\rangle =\langle \f(x)^{\k} \rangle =\langle \f(y)^{\l} \rangle$. Therefore, since $\langle \f(y)^{\l} \rangle $ is a subgroup of $  \langle \f(y)\rangle$ we have that $ \f(x) \in \langle \f(y)\rangle =K$, which is contradiction, because that would imply that $[\f(x),\f(y)]= 1$ since $K$ is cyclic. Therefore $ord(\f(x))\neq ord(\f(x)^{\k})$.  With similar arguments we can deduce that $ord(\f(y))\neq ord(\f(y)^{\l})$.

	Now we know that  $ord(\f(x)^{\k})=\dfrac{p^\nu}{gcd(\k,p^{\nu})}$ and $ord(\f(x))\neq ord(\f(x)^{\k})$. Therefore $gcd(\k,p^{\nu})\neq 1$ and thus $p$ divides $\k$. Similarly, $ord(\f(y)^{\l})=\dfrac{p^\mu}{gcd(\l,p^{\mu})}$ and $ord(\f(y))\neq ord(\f(y)^{\l})$ and thus $p$ divides $\l$. But we assumed that  $gcd(\k,\l)=1$ and hence we have a further contradiction. Therefore, we conclude that $[x,y]\in \gamma_{\omega}(G)$.
\end{proof}
\begin{corollary}\label{coprime}
		Let $G$ be a group and $a,b\in G$ with finite coprime orders. Then $[a,b]\in \gamma_{\omega}(G)$.
\end{corollary}

Let $A_{1},A_{2},\dots,A_{m}$ be finitely generated abelian groups with  $|\tau(A_{i})|=n_{i}$,  $i\in\{1,\dots,m\}$, where $\tau(A_{i})$ is the torsion subgroup of each group $A_{i}$. Thus $A_{i}=\mathbb{Z}^{N_{i}}\times\tau(A_{i})$,  $i\in\{1,\dots,m\}$.

 Each subgroup $\tau(A_{i})$ can be considered as the direct product $\mathbb{Z}_{p_{i,1}^{\k_{i,1}}}\times\mathbb{Z}_{p_{i,2}^{\k_{i,2}}}\times\dots\times\mathbb{Z}_{p_{i,r_{i}}^{\k_{i,r_{i}}}}$, where $p_{i,1}^{\k_{i,1}}p_{i,2}^{\k_{i,2}}\dots p_{i,r_{i}}^{\k_{i,r_{i}}}$ is the prime factorization of each $n_{i}$. Let $x_{i,j}$ be the generator
 of each direct factor in $\tau(A_{i})$,  and $t_{i,\bar{j}}$ be the generator of each torsion-free direct factor $\mathbb{Z}^{N_{i}}$  in the groups $A_{i}$, for all $i\in\{1,\dots,m\}$. Thus for each group
 $A_{i}$ we have $$A_{i}=\langle t_{i,1}\rangle \times \dots \langle t_{i,N_{1}} \rangle \times \langle x_{i,1}\hspace{1mm}|\hspace{1mm}x_{i,1}^{p_{i,1}^{\k_{i,1}}}\rangle \times \dots \langle x_{i,r_{i}}\hspace{1mm}|\hspace{1mm}x_{i,r_{i}}^{p_{i,1}^{\k_{i,r_{i}}}}\rangle$$ 
\begin{proposition}\label{freeprodabelian}	
	Let $G$ be the free product $G=A_{1}*A_{2}*\cdots *A_{m}$ of the  abelian groups $A_{i}$, $i\in\{1,\dots,m\}$. Then $\gamma_{\omega}(G)$ is the normal closure of the subset $\{[wx_{i_{1},j_{1}}w^{-1},x_{i_{2},j_{2}}] \}$ in $G$, where $i_{1},i_{2}\in\{1,\dots,m\}$ with $i_{1}\neq i_{2}$, $j_{1}\in\{1,\dots,r_{i_{1}}\}$, $j_{2}\in\{1,\dots,r_{i_{2}}\}$ and $p_{i_{1},j_{1}}\neq p_{i_{2},j_{2}}$ and $w$ is a  word in the generators $\{t_{i,\bar{j}}\}$, where $\bar{j}\in\{1,\dots,N_{i}\}$.
\end{proposition}
\begin{proof}
	Notice that the subgroup of $G$ generated by $\{t_{i,\bar{j}}\}$, $i\in\{1,\dots,m\}$, $\bar{j}\in\{1,\dots,N_{i}\}$ is a right-angled Artin group. Let $R_{t}$ be the set of relations of this group and let $S$ be  the subset $\{[wx_{i_{1},j_{1}}w^{-1},x_{i_{2},j_{2}}] \}$ of $G$. Also let $g$ be a commutator $[wx_{i_{1},j_{1}}w^{-1},x_{i_{2},j_{2}}]$. By Corollary \ref{coprime} we have that  $g$ belongs to $\gamma_{\omega}(G)$.   Therefore $S\subseteq \gamma_{\omega}(G)$.
	Now in order to prove that $\gamma_{\om}(G)=S^{G}$ we will prove that the group $\hat{G}=G/S^{G}$ is residually nilpotent.

	 The group $\hat{G}$ has the following presentation $$\hat{G}=\langle x_{i,j}, t_{i,\bar{j}} \hspace{1mm}|\hspace{1mm} x_{i,j}^{p_{i,j}^{\k_{i,j}}}, R_{t}, [wx_{i_{1},j_{1}}w^{-1},x_{i_{2},j_{2}}] \rangle$$
	where $\bar{j}\in\{1,\dots,N_{i}\}$,	$i,i_{1},i_{2}\in \{1,\dots,m\}$, 	$j_{\iota}\in \{1,\dots,r_{\iota}\}$, $\iota\in \{i,i_{1},i_{2}\}$ and if $i_{1}\neq i_{2}$, $p_{i_{1},j_{1}}\neq p_{i_{2},j_{2}}$ .
	
	Let $g$ be a non-trivial element in $\hat{G}$. If there is at least one $(i,\bar{j})$ such that the exponent sum of $t_{i,\bar{j}}$ is non-zero. Then we can take the homomorphism $\psi: \hat{G} \to \mathbb{Z}$ that maps $t_{i,\bar{j}}\mapsto 1$ and all other generators to zero then $\psi(g)$ is not trivial. Since $\mathbb{Z}$ is residually nilpotent the result follows.
	
	 On the other hand, assume that the exponent sum in each generator $t_{i,\bar{j}}$ in $g$ is zero. Then using the relations of $\hat{G}$,  $g$ can be written as a word $g=\chi_{1,1}\chi_{2,1}\dots\chi_{m,1}\cdot \chi_{1,2}\chi_{2,2}\dots \chi_{m,2}\cdot \chi_{1,\nu}\chi_{2,\nu}\dots \chi_{m,\nu}\cdot\chi$, where each $\chi_{i,\alpha}$ is word in conjugates of the generators $\{x_{i,j}\}$, $j\in\{1,\dots,r_{i}\}$, by words in $t_{i,\bar{j}}$, for all $i\in \{1,\dots,m\}$ and $\alpha\in \{1,\dots,\nu\}$ and $\chi$ is a word in the generators  $t_{i,\bar{j}}$.

	 If $\chi \neq 1$ then let $\psi$ be the homomorphism that maps all the generators of finite order to zero and the generators  $t_{i,\bar{j}}$ of infinite order to themselves. Then, $\psi(g)=\chi\neq 1$ and since the group generated by $t_{i,\bar{j}}$ is a right-angled Artin group it is residually nilpotent so the result follows.

	 Assume now that $\chi=1$.	 Let $\chi_{i,\alpha}=w _{i,1}^{\alpha}x_{i,1}^{\l _{i,1}^{\alpha}}(w _{i,1}^{\alpha})^{-1}\cdot w _{i,2}^{\alpha}x_{i,2}^{\l _{i,2}^{\alpha}}(w _{i,2}^{\alpha})^{-1}\dots w_{i,r_{i}}^{\alpha}x_{i,r_{i}}^{\l _{i,r_{i}}^{\alpha}}(w_{i,r_{i}}^{\alpha})^{-1}$, where $\l _{i,j}^{\alpha}\in \mathbb{Z}$ with $|\l _{i,j}^{\alpha}|<p_{i,j}^{\k_{i,j}}$ and $w_{i,j}^{\alpha}$ are words in $t_{i,\bar{j}}$.
	
	Notice that the relations $[wx_{i_{1},j_{1}}w^{-1},x_{i_{2},j_{2}}]$ of $\hat{G}$, where $w$ is any word in the generators $\{t_{i,\bar{j}}\}$, are equivalent to the relations $[\hat{w}x_{i_{1},j_{1}}\hat{w}^{-1},\tilde{w}x_{i_{2},j_{2}}\tilde{w}^{-1}]$, where $\hat{w}$ and $\tilde{w}$ are all in the generators $\{t_{i,\bar{j}}\}$. Using these relations, each subword $\chi_{1,\alpha}\chi_{2,\alpha}\dots \chi_{m,\alpha}$ which is equal to $w _{1,1}^{\alpha}x_{1,1}^{\l _{1,1}^{\alpha}}(w _{1,1}^{\alpha})^{-1}\cdot w _{1,2}^{\alpha}x_{1,2}^{\l _{1,2}^{\alpha}}(w _{1,2}^{\alpha})^{-1}\dots w _{1,r_{1}}^{\alpha}x_{1,r_{1}}^{\l _{1,r_{1}}^{\alpha}}(w _{1,r_{1}}^{\alpha})^{-1}\cdot w _{2,1}^{\alpha}x_{2,1}^{\l _{2,1}^{\alpha}}(w _{2,1}^{\alpha})^{-1}\cdot w _{2,2}^{\alpha}x_{2,2}^{\l _{2,2}^{\alpha}}(w _{2,2}^{\alpha})^{-1}\dots$ $ w _{2,r_{2}}^{\alpha}x_{2,r_{2}}^{\l _{2,r_{2}}^{\alpha}}(w _{2,r_{2}}^{\alpha})^{-1}$ $\dots w _{m,1}^{\alpha}x_{m,1}^{\l _{m,1}^{\alpha}}(w _{m,1}^{\alpha})^{-1}\cdot$ \\ $ w _{m,2}^{\alpha}x_{m,2}^{\l _{m,2}^{\alpha}}(w _{m,2}^{\alpha})^{-1}\dots w _{m,r_{m}}^{\alpha}x_{m,r_{m}}^{\l _{m,r_{m}}^{\alpha}}(w _{m,r_{m}}^{\alpha})^{-1}$ can be written as 
$\bar{\chi}_{p_{1},\alpha}\bar{\chi}_{p_{2},\alpha}\dots\bar{\chi}_{p_{s},\alpha}$, where each $\bar{\chi}_{p_{i},\alpha}$ is a word that consists only of conjugates of generators with orders powers of the prime number $p_{i}$ and $p_{1},\dots,p_{s}$ are the distinct prime numbers that appear in the prime factorizations of $n_{i}$ at least twice.

Now $g=\bar{\chi}_{p_{1},1}\bar{\chi}_{p_{2},1}\dots\bar{\chi}_{p_{s},1}\cdot \bar{\chi}_{p_{1},2}\bar{\chi}_{p_{2},2}\dots\bar{\chi}_{p_{s},2}\dots \bar{\chi}_{p_{1},\nu}\bar{\chi}_{p_{2},\nu}\dots\bar{\chi}_{p_{s},\nu}$.
 Again using the relators $[wx_{i_{1},j_{1}}w^{-1},x_{i_{2},j_{2}}]$ of $\hat{G}$, we have that $[\bar{\chi}_{p_{i},\alpha_{1}},\bar{\chi}_{p_{j},\alpha_{2}}]=1$ for $\alpha_{1}\neq \alpha_{2}\in \{1,\dots,\nu\}$ and $i\neq j\in \{1,\dots,s\}$. Hence $g$ can be written as the word $$g=\bar{\chi}_{p_{1},1}\bar{\chi}_{p_{1},2}\dots\bar{\chi}_{p_{1},\nu}\cdot \bar{\chi}_{p_{2},1}\bar{\chi}_{p_{2},2}\dots\bar{\chi}_{p_{2},\nu}\dots \bar{\chi}_{p_{s},1}\bar{\chi}_{p_{s},2}\dots\bar{\chi}_{p_{s},\nu}.$$
 Finally, let $W_{p_{j}}$ be the word $\bar{\chi}_{p_{j},1}\bar{\chi}_{p_{j},2}\dots\bar{\chi}_{p_{j},\nu}$, for every $j\in \{1,\dots,s\}$ and thus $g={W}_{p_{1}}{W}_{p_{2}}\dots {W}_{p_{s}}$.
 
   Notice that each subword ${W}_{j}$ consists of conjugates of generators $x_{i,j}$, $i\in \{1,\dots,m\}$, which they all have order a power of the same prime number. Moreover, since $g\neq1$, at least one of the reduced words ${W}_{p_1}, {W}_{p_2}, \dots , {W}_{p_s}$ is not trivial.
	
	Let $\beta$ be an index such that ${W}_{p_{\beta}}$ is not trivial and let $\phi_{\beta}$ be the homomorphism $\phi_{\beta} : \hat{G} \to *(\mathbb{Z}_{N_{i}}\times \mathbb{Z}_{p_{\beta}^{\k_{i,j}}})=H_{\beta}$, where $\mathbb{Z}_{p_{\beta}^{\k_{i,j}}}$ are the groups generated by the generators $x_{i,j}$ that belong to the word ${W}_{p_{\beta}}$. Hence $\phi_{\beta}(\langle t_{i,\bar{j}}\rangle)=\mathbb{Z}$, for all $i\in \{1,\dots,m\}$, $\bar{j}\in \{1,\dots,N_{i}\}$ and $\phi_{\beta}(\langle x_{i,j}\rangle)=\mathbb{Z}_{p_{\beta}^{\k_{i,j}}}$ if $ord(x_{i,j})=p_{\beta}^{\k_{i,j}}$, while $\phi_{\beta}(x_{i,j})=0$ otherwise. Then $\phi_{\beta}(g)\neq 1$. Now since the group $H_{\beta}$  is residually finite $p$-group and thus residually nilpotent, the result follows.\end{proof}

Assume now we have two finite abelian groups $A$ and $B$ with $|A|=k$ and $|B|=l$. Then $A$ can be considered as the direct product (we rearrange appropriately) $$A\cong\mathbb{Z}_{p_{1}^{\k_{1}}}\times\dots\times \mathbb{Z}_{p_{s}^{\k_{s}}}\times \mathbb{Z}_{q_{s+1}^{\k_{s+1}}} \dots\times\mathbb{Z}_{q_{n}^{\k_{n}}}$$ and $B$ as the direct product $$B\cong\mathbb{Z}_{p_{1}^{\l_{1}}}\times\dots\times \mathbb{Z}_{p_{s}^{\l_{s}}}\times \mathbb{Z}_{r_{s+1}^{\l_{s+1}}} \dots\times\mathbb{Z}_{r_{n}^{\l_{m}}}$$ where $p_{1}^{\k_{1}}\dots p_{s}^{\k_s}q_{s+1}^{\k_{s+1}}\dots q_{n}^{\k_{n}}$ and  $p_{1}^{\l_{1}}\dots p_{s}^{\l_s}r_{s+1}^{\l_{s+1}}\dots r_{m}^{\l_{m}}$ are the respective prime number factorizations of $n$ and $m$. Let $x_{i}$ be the generator of each direct factor of $A$, $i\in \{1,\dots,n\}$ and  $y_{j}$ be the generator of each direct factor of $B$, $j\in \{1,\dots,m\}$. Then the above Proposition gives us the following result. 
\begin{corollary}\label{Cor1}
	Let $G$ be the free product $G=A*B$ of finite abelian groups $A$ and $B$ described above. Then $\gamma_{\omega}(G)$ is the normal closure of the subset $\{[x_{i},y_{j}] \}$ in $G$, where   $i\in\{1,\dots,n\}$, $j\in\{1,\dots,m\}$ and if $i,j\leq s$ then $i\neq j$.
\end{corollary}
\begin{remark*}
If $A=\< a \hspace{1mm}|\hspace{1mm}a^{k}=1\>$ and $B=\<b\hspace{1mm}|\hspace{1mm}b^{l}=1\>$ then we can use the following isomorphisms $\phi_{1}$ and $\phi_{2}$ in order to write  $\gamma_{\omega}(G)$ in terms of the generators $a$ and $b$. 
\vspace{2mm}

$  \begin{array}{rcl}
	\f_{1}: \mathbb{Z}_{p_{1}^{\k_{1}}}\times\dots\times \mathbb{Z}_{p_{s}^{\k_{s}}}\times \mathbb{Z}_{q_{s+1}^{\k_{s+1}}} \dots\times\mathbb{Z}_{q_{n}^{\k_{n}}} &  \xrightarrow{\cong} & \mathbb{Z}_{k} \\ x_{i} &
	\mapsto & a^{\frac{k}{p_{i}^{\k_{i}}}}  \hspace{2mm} \text{ if } \hspace{2mm}  1\leq i\leq s \\ x_{i} &
	\mapsto & a^{\frac{k}{q_{i}^{\k_{i}}}} \hspace{2mm} \text{ if } \hspace{2mm}  s+1\leq i\leq n
\end{array}   $

\vspace{3mm}

$  \begin{array}{rcl}
	\f_{2}: \mathbb{Z}_{p_{1}^{\l_{1}}}\times\dots\times \mathbb{Z}_{p_{s}^{\l_{s}}}\times \mathbb{Z}_{r_{s+1}^{\l_{s+1}}} \dots\times\mathbb{Z}_{r_{n}^{\l_{m}}} & \xrightarrow{\cong} & \mathbb{Z}_{l} \\ y_{i} &
	\mapsto & b^{\frac{l}{p_{i}^{\l_{i}}}}  \hspace{2mm} \text{ if } \hspace{2mm}  1\leq i\leq s \\ y_{i} &
	\mapsto & b^{\frac{l}{r_{i}^{\l_{i}}}} \hspace{2mm} \text{ if } \hspace{2mm}  s+1\leq i\leq m
\end{array}   $

\vspace{3mm}

Let $\bar{G}$ be the group $\f_{1}^{-1}(\mathbb{Z}_{k})*\f_{2}^{-1}(\mathbb{Z}_{l})$. Then $\bar{G}$ is isomorphic to $G$ by the isomorphism $\f$ that maps each letter of any reduced word in $x_{i},y_{i}$ using $\f_{1}$ or $\f_{2}$.

Let $S=\{ [a^{c_{i}},b^{d_i}],[b^{e_i},a^{f_i}]\}$ in $G$, where $c_{i}=p_{i}^{\k_{i}}$, $d_{i}=p_{1}^{\l_{1}}\dots p_{i-1}^{\l_{i-1}}p_{i+1}^{\l_{i+1}}\dots p_{s}^{\l_s}$ and $e_{i}=p_{i}^{\l_{i}}$, $f_{i}=p_{1}^{\k_{1}}\dots p_{i-1}^{\k_{i-1}}p_{i+1}^{\k_{i+1}}\dots p_{s}^{\k_s}$ for all $i\in \{1,\dots s\}$. Also, let  $\xi_{k}=p_{1}^{\k_{1}}\dots p_{s}^{\k_s}$ and $\xi_{l}=p_{1}^{\l_{1}}\dots p_{s}^{\l_s}$.

We know that $\gamma_{\om}(\bar{G})$ is the normal closure of the subset $\bar{S}=\{[x_{i},y_{j}] \}$, where   $i\in\{1,\dots,n\}$, $j\in\{1,\dots,m\}$ and if $i,j\leq s$ then $i\neq j$.
We will prove that the image of $\bar{S}$ under the isomorphism $\f$ is the set $S$.

Let $i\in \{s+1,\dots,n\}$ fixed. Then $[x_{i},y_{j}]\in \gamma_{\om}(\bar{G})$ for all $j\in \{1,\dots,m\}$, which means that $[a^{\frac{k}{q_{i}^{\k_{i}}}}, b^{\frac{l}{p_{j}^{\l_{j}}}}]\in \gamma_{\om}(G)$ for all $j\in \{1,\dots,s\}$ and also $[a^{\frac{k}{q_{i}^{\k_{i}}}}, b^{\frac{l}{r_{j}^{\l_{j}}}}]\in \gamma_{\om}(G)$ for all $j\in \{s+1,\dots,m\}$.  Now since $$gcd\big(\frac{l}{p_{1}^{\l_{1}}},\frac{l}{p_{2}^{\l_{2}}},\dots,\frac{l}{p_{s}^{\l_{s}}},\frac{l}{r_{s+1}^{\l_{s+1}}},\frac{l}{r_{s+2}^{\l_{s+2}}},\dots,\frac{l}{r_{m}^{\l_{m}}}\big)=1$$ these relations can be reduced to the  relation $[a^{\frac{k}{q_{i}^{\k_{i}}}}, b]\in \gamma_{\om}(G)$. Now, since this relation holds for every $i\in \{s+1,\dots,n\}$ and $gcd\big(\frac{k}{q_{s+1}^{\k_{s+1}}},\dots,\frac{k}{q_{n}^{\k_{n}}}\big)=\xi_{k}$ all these relations can again be reduced to the relation $[a^{\xi_{k}},b]\in \gamma_{\omega}(G)$. 

With similar arguments the relationships $[x_{i},y_{j}]\in \gamma_{\om}(\bar{G})$ where, $i\in \{1,\dots,n\}$ and $j\in \{s+1,\dots,m\}$ can be reduced to the relation $[b^{\xi_{l}},a]\in \gamma_{\om}(G)$.
\vspace{2mm}

Now let   $i\in \{1,\dots,s\}$ fixed.  Then $[x_{i},y_{j}]\in \gamma_{\om}(\bar{G})$ for all $j\in \{1,\dots,i-1,i+1,\dots,m\}$ ,  which means that $[a^{\frac{k}{p_{i}^{\k_{i}}}}, b^{\frac{l}{p_{j}^{\l_{j}}}}]\in \gamma_{\om}(G)$ for all $j\in \{1,\dots,i-1,i+1,\dots,s\}$ and $[a^{\frac{k}{p_{i}^{\k_{i}}}}, b^{\frac{l}{r_{j}^{\l_{j}}}}]\in \gamma_{\om}(G)$ for all $j\in \{s+1,\dots,m\}$.
\vspace{3mm}

 It holds that  $$gcd\big(\frac{l}{p_{1}^{\l_{1}}},\dots,\frac{l}{p_{i-1}^{\l_{i-1}}},\frac{l}{p_{i+1}^{\l_{i+1}}},\dots,\frac{l}{p_{s}^{\l_{s}}},\frac{l}{r_{s+1}^{\l_{s+1}}},\dots,\frac{l}{r_{m}^{\l_{m}}}\big)=1.$$ Therefore, these relations can be reduced to the relation $[a^{\frac{k}{p_{i}^{\k_{i}}}}, b^{p_{i}^{\l_{i}}}]\in \gamma_{\om}(G)$. Thus, for every $i\in \{s+1,\dots,n\}$ we have that $[a^{\frac{k}{p_{i}^{\k_{i}}}}, b^{p_{i}^{\l_{i}}}]\in \gamma_{\om}(G)$. But since $[a^{\xi_{k}},b]\in \gamma_{\omega}(G)$ and $gcd(\frac{k}{p_{i}^{\k_{i}}},\xi_{k})=\frac{\xi_{k}}{p_{i}^{\k_{i}}}$ we have that  $[a^{\frac{\xi_{k}}{p_{i}^{\k_{i}}}}, b^{p_{i}^{\l_{i}}}]\in \gamma_{\om}(G)$. 
 
 With similar arguments we have that  for all $j\in \{1,\dots,s\}$,  $[b^{\frac{\xi_{l}}{p_{i}^{\l_{i}}}}, a^{p_{i}^{\k_{i}}}]\in \gamma_{\om}(G)$. 
 
 Since every relation is now a consequence of the relations $[a^{\frac{\xi_{k}}{p_{i}^{\k_{i}}}}, b^{p_{i}^{\l_{i}}}]\in \gamma_{\om}(G)$ and $[b^{\frac{\xi_{l}}{p_{i}^{\l_{i}}}}, a^{p_{i}^{\k_{i}}}]\in \gamma_{\om}(G)$ we deduce that $\gamma_{\om}(G)=S^{G}$.
\end{remark*}

\section{Calculation of $\gamma_{\om}(G)$ for GBS tree-groups}
\vspace{2mm}
\subsection{The segment}\label{segm}
We will now use Lemma \ref{Lemma1} and Corollary \ref{Cor1} in order to calculate the intersection $\gamma_{\omega}(G)$, where $G$ is the amalgamated free product of two infinite cyclic groups, that is a GBS group where the underlying graph is a segment. We will first provide a proof to determine explicitly when such a group is residually nilpotent, that is $\gamma_{\omega}(G)$ is trivial. This result was first proved by McCarron in \cite{McC}.
\vspace{1mm}

\begin{corollary}
	Let $G$ be the free product of the infinite cyclic groups $\langle a\rangle$ and $\langle b\rangle$ amalgamated over the cyclic subgroups $\langle a^{k}\rangle$ and $\langle b^{l}\rangle$ where $gcd(k,l)=1$. Then $\gamma_{\omega}(G)=G'$.
\end{corollary}
	\begin{proof}
		By Lemma \ref{Lemma2}, since $gcd(k,l)=1$ we have that $[a,b]\in \gamma_{\omega}(G)$. Moreover since \\${G}/{\langle [a,b] \rangle}={G}/{G'}=G^{ab}$ is abelian (and thus nilpotent) we have that $\gamma_{\omega}(G)=G'$.
	\end{proof}

\begin{remark*}
If $k,l$ are not powers of the same prime number then the above Corollary states that $G$ is not residually nilpotent.
\end{remark*}

\begin{proposition}\label{prop8}
Let $G$ be the  free product of the infinite cyclic groups $\langle a\rangle$ and $\langle b\rangle$ amalgamated over the cyclic subgroups $\langle a^{k}\rangle$ and $\langle b^{l}\rangle$. If $k,l$ are  powers of the same prime number $p$ then $G$ is residually a finite $p$-group.
\end{proposition}
\begin{proof}
We have that $G$ is the amalgamated free product $G=G_{1}*_{A=B} G_{2}$, where $G_{1}=\langle a \rangle$, $G_{2}=\langle b  \rangle$ and $A,$ $ B$ are the subgroups $\langle a^{p^\k}\rangle$ and $\langle b^{p^\l}\rangle$ respectively, where $\k,\l\in \mathbb{N}$ and $p$ is a prime number. 
	
	Let $G^{ab}={G}/{G'}$ be the abelianization of $G$. Thus $$G^{ab}=\langle a,b \hspace{1mm}|\hspace{1mm} a^{p^\k}=b^{p^\l}, [a,b]=1 \rangle.$$ We have that $G^{ab}$ is residually a finite $p$-group. Indeed, assume without loss of generality $\k\leq l$ and let $x=ab^{-p^{\l-\k}}$ and $y=b$. Then $G^{ab}$ has a new presentation 
	$$G^{ab}=\langle a,b \hspace{1mm}|\hspace{1mm}  x^{p^\k}y^{p^\k p^{\l-\k}}=y^{p^\l}, [x,y]=1 \rangle.$$
	Now the relation $x^{p^\k}y^{p^\k p^{\l-\k}}=y^{p^\l}$ gives us $x^{p^\k}=1$ and thus  $G^{ab}\cong \mathbb{Z}\times\mathbb{Z}_{p^\k}$ which is residually a finite $p$-group .

	Now let $N=\langle a^{p^\k}\rangle=\langle b^{p^\l}\rangle$ (which is in fact the center of $G$). Then $N$ is normal in $G$ and ${G}/{N}\cong \mathbb{Z}_{p^\k}* \mathbb{Z}_{p^\l}$, which is residually a finite $p$-group from Gruenberg's result. 
	
	Let $g\in G$ with $g\neq 1$ a reduced word. Let $\f_{1}$ be the epimorphism $\f_{1}:G \to {G^{ab}}$ and $\f_{2}$ be the epimorphism  $\f_{2}:G \to {G}/{N}$. If the exponent sum of $a$ or $b$ in $g$ is different than zero then $\f_{1}(g)\neq 1$. Otherwise $g$ can be written as a product of commutators and thus $g\in \gamma_{2}(G)$. But since $\gamma_{2}(G)\cap Z(G)$ is trivial (by Lemma \ref{intersectionGBS})  we have that $\f_{2}(g)\neq 1$. Therefore, for all $g\neq1$ we have a homomorphism form $G$ to residually a finite $p$-group such that the image of $g$ is not trivial. Consequently, $G$ is residually a finite $p$-group.
\end{proof}
\begin{corollary}\label{cor3}
Let $G$ be the free product of the infinite cyclic groups $\langle a\rangle$ and $\langle b\rangle$ amalgamated over the cyclic subgroups $\langle a^{k}\rangle$ and $\langle b^{l}\rangle$. Then $G$ is residually nilpotent if and only if $k,l$ are powers of the same prime number.
\end{corollary}
\begin{proof}
	The corollary follows from the fact that every residually finite $p$-group is residually nilpotent.
\end{proof}
\vspace{1mm}
The above result determines the case where the amalgamated free product $G$ of two infinite cyclic groups in residually nilpotent and thus $\gamma_{\omega}(G)=\{1\}$.
We will now calculate the intersection $\gamma_{\omega}(G)$ in the general case.

\begin{theorem}\label{amalgamatedprod}
Let $G$ be the free product of the infinite cyclic groups $\langle a\rangle$ and $\langle b\rangle$ amalgamated over the cyclic subgroups $\langle a^{k}\rangle$ and $\langle b^{l}\rangle$, where $k=p_{1}^{\k_{1}}\dots p_{s}^{\k_s}q_{s+1}^{\k_{s+1}}\dots q_{n}^{\k_{n}}$ and  $l=p_{1}^{\l_{1}}\dots p_{s}^{\l_s}r_{s+1}^{\l_{s+1}}\dots r_{m}^{\l_{m}}$ are their respective prime factorizations. Then $\gamma_{\omega}(G)$ is the normal closure of the subset \\ $\big{\{} [a^{c_{i}},b^{d_i}],[b^{e_i},a^{f_i}]\big{\}}$ in $G$, where $c_{i}=p_{i}^{\k_{i}}$, $d_{i}=p_{1}^{\l_{1}}\dots p_{i-1}^{\l_{i-1}}p_{i+1}^{\l_{i+1}}\dots p_{s}^{\l_s}$ and $e_{i}=p_{i}^{\l_{i}}$,\\ $f_{i}=p_{1}^{\k_{1}}\dots p_{i-1}^{\k_{i-1}}p_{i+1}^{\k_{i+1}}\dots p_{s}^{\k_s}$ for all $i\in \{1,\dots s\}$.
\end{theorem}
\begin{proof}

 By Lemma \ref{Lemma1} we  know that $\gamma_{\om}(G)\cong\gamma_{\om}(G/Z(G))=\gamma_{\om}(\mathbb{Z}_{k}*\mathbb{Z}_{l})$.
  Now if $H$ is the free product $H=\mathbb{Z}_{k}*\mathbb{Z}_{l}$ using Corollary \ref{Cor1} (and the remark) for $H$ we have the result.
\end{proof}
\vspace{2mm}
\subsection{The case of the tree}\label{tree}
\vspace{1mm}
We will now calculate $\gamma_{\omega}(G)$ in the case where the underlying graph of the generalized Baumslag-Solitar group $G$ is a tree $T$  based on the calculation of $\gamma_{\om}(G)$ in the case where the graph is a segment. More specifically, we prove that in order to calculate $\gamma_{\om}(G)$ it is sufficient to calculate the intersection $\gamma_{\om}$ for each subgroup of $G$ that is defined by a path in the tree (using Theorem \ref{amalgamatedprod}) and  take their union. 

Remember that since $T$ is a tree, we have that $Z(G)$ is a non-trivial cyclic subgroup of each vertex group of $G$. (see \cite{Delga})
\vspace{1mm}
\begin{theorem}
Let $G$ be the GBS group presented by the following labeled tree $T$ 
and let $H_{(i,j)}^{(i',j')}$ be the subgroup generated by the pair of distinct vertex groups  $\<x_{i,j}\>$ and  $\<x_{i',j'}\>$. Let  $S_{(i,j)}^{(i',j')}$ be the set defined in Theorem \ref{amalgamatedprod} for each such subgroup and let $S$ be their union. Then $\gamma_{\omega}(G)$ is the normal closure of $S$ in $G$.

\begin{center}
	\begin{tikzpicture}[scale=0.97]
		
		\draw[thick,-] (1.5,0) -- (4,0) node[midway,sloped,above] {\footnotesize{$n_{1,1}\hspace{6mm} m_{1,1}$}};
		\draw[ thick,-] (4,0) -- (8,0) node[midway,sloped,above] {\footnotesize{$n_{2,1}\hspace{11mm}m_{2,1}$}};
		\draw[ thick,-] (4,0) -- (8,2.5) node[midway,sloped,above] {\footnotesize{$n_{2,2}\hspace{18mm}m_{2,2}$}};
		\draw[ thick,-] (4,0) -- (8,-2.5) node[midway,sloped,above] {\footnotesize{$n_{2,3}\hspace{18mm}m_{2,3}$}};
		\draw[thick,-] (8,0) -- (12,0)node[midway,sloped,above] {\footnotesize{$n_{3,1}\hspace{5mm}m_{3,1}$}};
		\draw[thick,-] (8,0) -- (12,1.7)node[midway,sloped,above] {\footnotesize{$n_{3,2}\hspace{10mm}m_{3,2}$}};
		\draw[thick,-] (8,0) -- (12,-1.7)node[midway,sloped,above] {\footnotesize{$n_{3,3}\hspace{10mm}m_{3,3}$}};
		\draw[thick,-] (8,2.5) -- (12,2)node[midway,sloped,above] {\footnotesize{$n_{3,4}\hspace{10mm}m_{3,4}$}};
		\draw[thick,-] (8,2.5) -- (12,3.8)node[midway,sloped,above] {\footnotesize{$n_{3,6}\hspace{10mm}m_{3,6}$}};
		\draw[thick,-] (8,-2.5) -- (12,-2)node[midway,sloped,above] {\footnotesize{$n_{3,5}\hspace{10mm}m_{3,5}$}};
		\draw[thick,-] (8,-2.5) -- (12,-3.8)node[midway,sloped,above] {\footnotesize{$n_{3,7}\hspace{10mm}m_{3,7}$}};
		\fill (1.5cm,0cm) circle (1.5pt);  \fill (4cm,0cm) circle (1.5pt);
		\fill (8cm,2.5cm) circle (1.5pt);\fill (8cm,-2.5cm) circle (1.5pt);
		\fill (8cm,0cm) circle (1.5pt); \fill (12cm,1.7cm) circle (1.5pt); \fill (12cm,-1.7cm) circle (1.5pt);\fill (12cm,0cm) circle (1.5pt);\fill (12cm,-2cm) circle (1.5pt);\fill (12cm,2cm) circle (1.5pt);\fill (12cm,-3.8cm) circle (1.5pt);\fill (12cm,3.8cm) circle (1.5pt);
		\fill (13cm,0cm) circle (1pt);  \fill (13.5cm,0cm) circle (1pt); \fill (14cm,0cm) circle (1pt);
		\fill (9.8cm,3.9cm) circle (1pt);  \fill (9.8cm,4.2cm) circle (1pt); \fill (9.8cm,4.5cm) circle (1pt);
		\fill (9.8cm,1.5cm) circle (1pt);  \fill (9.8cm,1.7cm) circle (1pt); \fill (9.8cm,1.9cm) circle (1pt);
		\fill (9.8cm,-1.1cm) circle (1pt);  \fill (9.8cm,-1.3cm) circle (1pt); \fill (9.8cm,-1.5cm) circle (1pt);
		\fill (9.8cm,-3.5cm) circle (1pt);  \fill (9.8cm,-3.8cm) circle (1pt); \fill (9.8cm,-4.1cm) circle (1pt);
		\fill (5.5cm,2cm) circle (1pt);  \fill (5.5cm,2.3cm) circle (1pt); \fill (5.5cm,2.6cm) circle (1pt);
		\fill (5.5cm,-1.6cm) circle (1pt);  \fill (5.5cm,-1.9cm) circle (1pt); \fill (5.5cm,-2.2cm) circle (1pt);
		\node at (1.5,-0.4) (nodeA) {\large{$\langle x_{1,1}\rangle$}}; 
		\node at (4,-0.4) (nodeB) {\large{$\langle x_{2,1}\rangle$}};
		\node at (8,-0.4) (nodeP) {\large{$\langle x_{3,1}\rangle$}};
		\node at (8,-2.9) (nodeP) {\large{$\langle x_{3,3}\rangle$}};
		\node at (8,2.9) (nodeP) {\large{$\langle x_{3,2}\rangle$}};
		\node at (12,-0.4) (nodeP) {\large{$\langle x_{4,1}\rangle$}};
		\node at (12.7,-1.6) (nodeP) {\large{$\langle x_{4,3}\rangle$}};
		\node at (12.7,1.4) (nodeP) {\large{$\langle x_{4,2}\rangle$}};
		\node at (12.7,-2.2) (nodeP) {\large{$\langle x_{4,5}\rangle$}};
		\node at (12.7,-3.9) (nodeP) {\large{$\langle x_{4,7}\rangle$}};
		\node at (12.7,2.1) (nodeP) {\large{$\langle x_{4,4}\rangle$}};
		\node at (12.7,3.7) (nodeP) {\large{$\langle x_{4,6}\rangle$}};

	\end{tikzpicture} 

\end{center} 
  
\end{theorem} 
\begin{proof}

 Assume that every vertex group in $\tilde{G}$ associated to each vertex $v_{i,j}$ is the finite cyclic group $\< x_{i,j}\hspace{1mm}|\hspace{1mm}x_{i,j}^{d_{i,j}}=1\>\cong \mathbb{Z}_{d_{i,j}}$.
	The center of $G$ is $Z(G)=\langle x_{i,j}^{d_{i,j}} \rangle$ , where $n_{i,{j}},m_{i-1,j} | d_{i,j}$. 
	By Lemma \ref{Lemma1}, we have that $\gamma_{\om}(G)$ is isomorphic to $\gamma_{\om}(G/Z(G))$. 	Let $\tilde{G}$ be the group $G/Z(G)$.
	 
	Let $ p_{1_{(i,j)}}^{k_{1_{(i,j)}}} p_{2_{(i,j)}}^{k_{2_{(i,j)}}}\dots  p_{n_{(i,j)}}^{k_{n_{(i,j)}}}$ be the prime factorization of each $d_{i,j}$.  For each vertex group  $\<x_{i,j}\>$ in $\tilde{G}$ let $f_{(i,j)}$ be the following isomorphism.

	$  \begin{array}{rcl}
		f_{(i,j)}: \mathbb{Z}_{d_{i,j}} &  \xrightarrow{\cong} & \mathbb{Z}_{ p_{1_{(i,j)}}^{k_{1_{(i,j)}}}}\times\mathbb{Z}_{p_{2_{(i,j)}}^{k_{2_{(i,j)}}}}\times\dots\times \mathbb{Z}_{p_{n_{(i,j)}}^{k_{n_{(i,j)}}}} \vspace{2mm} \\  x_{i,j} &
		\mapsto & y_{1_{(i,j)}}^{\l_{1_{(i,j)}}} y_{2_{(i,j)}}^{\l_{2_{(i,j)}}}\dots  y_{n_{(i,j)}}^{\l_{n_{(i,j)}}} 
		
	\end{array}   $ \vspace{2mm}
\\
where $ {\l_{1_{(i,j)}}}, {\l_{2_{(i,j)}}},\dots, {\l_{n_{(i,j)}}}$ are the Bezout coefficients of the integers $\frac{d_{i,j}}{p_{\a_{(i,j)}}}$ where $\a=1,\dots, n$,   with $$gcd(\frac{d_{i,j}}{p_{1_{(i,j)}}},\frac{d_{i,j}}{p_{1_{(i,j)}}},\dots,\frac{d_{i,j}}{p_{n_{(i,j)}}})=1$$ 
for each vertex group $\<x_{i,j}\>$.

Let $\hat{G}$ be the group isomorphic to $\tilde{G}$ where each $x_{i,j}$ is replaced using the  isomorphisms $	f_{(i,j)}$. The generating set of $\hat{G}$ is the set $\{y_{\a_{(i,j)}}\}$, for all $\a\in \{1,\dots,n\}$ and all $(i,j)$ with $v_{i,j}\in V(T)$.

 The relations of $\hat{G}$ are the relations $y_{\a_{(i,j)}}^{ p_{\a_{(i,j)}}^{k_{\a_{(i,j)}}}}=1$, for all $\a\in \{1,\dots,n\}$ and all $(i,j)$ with $v_{i,j}\in V(T)$, the relations $[y_{\a_{(i,j)}},y_{\b_{(i,j)}}]$,   for all $\a,\b\in \{1,\dots,n\}$ and all $(i,j)$ with $v_{i,j}\in V(T)$ and the amalgamating relations $f_{(i,j)}(x_{i,j})^{n_{i,j}}=f_{(i+1,j')}(x_{i+1,j'})^{m_{i,j}}$.

Let $X$ be the subset of $\hat{G}$ consisting of the commutators $[y_{\a_{(i,j)}},y_{\b_{(i',j')}}]$, where $(i,j)\neq(i',j')$ and $p_{\a_{(i,j)}}\neq p_{\b_{(i',j')}}$, i.e the set $X$ contains all the commutators between the generators of $\hat{G}$ with coprime orders (some of which may be trivial).  By Corollary \ref{coprime} we have that $X$ is contained in $\gamma_{\om}(\hat{G})$.
We will prove that the normal closure of the set $X$ is $\gamma_{\om}(\hat{G})$ by proving that the group $H=\hat{G}/X^{\hat{G}}$ is residually nilpotent.

	Let $1\neq g\in H$. Using the procedure described in Lemma \ref{freeprodabelian} we can write $g$ as   $g={W}_{p_{1}}{W}_{p_{2}}\dots{W}_{p_{r}}$, where $p_{1},p_{2},\dots, p_{r}$ are the distinct prime numbers that arise in the prime factorizations of the exponents $d_{i,j}$ and each word  
	 ${W}_{p_{i}}$ is a word that consists only of  generators with orders powers of the prime number $p_{i}$. Now since $g\neq 1 $, at least one of the words ${W}_{p_{1}}, {W}_{p_{2}},\dots, {W}_{p_{r}}$ is not trivial. Let $W_{p_{\iota}}$ be such a word, with $\iota\in \{1,\dots,r\}$.

	Let $C$ be the subset of the generating set of $H$ that consists of generators $y_{\a_{(i,j)}}$ with $ord(y_{\a_{(i,j)}})$ be a power of the prime number $p_{\iota}$. Let $R_{C}$ be the amalgamating relations $f_{(i,j)}(x_{i,j})^{n_{i,j}}=f_{(i+1,j')}(x_{i+1,j'})^{m_{i,j}}$, in which if a generator $y_{\a_{(i,j)}}$   does not belong in $C$ then $y_{\a_{(i,j)}}=1$. The relations of $R_{C}$ are of the form $y_{\a_{(i,j)}}^{\mu_{\a_{(i,j)}}}=y_{\b_{(i+1,j')}}^{\nu_{\b_{(i+1,j')}}}$, with $y_{\a_{(i,j)}}, y_{\b_{(i+1,j')}}$ in $C$.

	 We define $G_{p_{\iota}}$ to be the group with generators $\bar{y}_{\a_{(i,j)}}$ for all $\a, (i,j)$ with $y_{\a_{(i,j)}}\in C$.  The relations of $G_{p_{\iota}}$   are the relations $\bar{y}_{\a_{(i,j)}}^{ p_{\a_{(i,j)}}^{k_{\a_{(i,j)}}}}=1$ ($p_{\a_{(i,j)}}=p_{\iota}$) as well as the relations $\bar{y}_{\a_{(i,j)}}^{\mu_{\a_{(i,j)}}}=\bar{y}_{\b_{(i+1,j')}}^{\nu_{\b_{(i+1,j')}}}$, with $y_{\a_{(i,j)}}, y_{\b_{(i+1,j')}}$ belong to $C$.

		Let $\f_{\iota}$ be the homomorphism $\f_{\iota} : H \to G_{p_{\iota}}$ 
		with $\f_{\iota}(y_{\a_{(i,j)}})=\bar{y}_{\a_{(i,j)}}$, for all generators $y_{\a_{(i,j)}}$ that belong to $C$ and $\f_{\iota}(y_{\a_{(i,j)}})=1_{G_{p_{\iota}}}$ for the rest of the generators of $H$. The map $\f_{\iota}$ is indeed a homomorphism since it preserves the relations of $H$. Since $W_{p_{\iota}}\neq 1$ we have that $\phi_{\iota}(W_{p_{\iota}})\neq 1$ and hence $\phi_{\iota}(g)\neq 1$.

		 The group $G_{p_{\iota}}$ is the free product of fundamental groups of subtrees of $T$ with finite cyclic $p_{\iota}$-groups, which by Proposition \ref{ptree} are residually finite $p_{\iota}$-groups and therefore their free product is also  residually a finite $p_{\iota}$-group and hence residually nilpotent.

		 Therefore the normal closure of the set $X$ is in $\gamma_{\om}(\hat{G})$. Now using the inverses of the isomorphisms $f_{(i,j)}$ (as described below) we can calculate $\gamma_{\om}(\bar{G})$ which is equal to  $\gamma_{\om}({G})$.
		 \vspace{2mm}

		 	$  \begin{array}{rcl}
		 	f_{(i,j)}^{-1}:\mathbb{Z}_{ p_{1_{(i,j)}}^{k_{1_{(i,j)}}}}\times\mathbb{Z}_{p_{2_{(i,j)}}^{k_{2_{(i,j)}}}}\times\dots\times \mathbb{Z}_{p_{n_{(i,j)}}^{k_{n_{(i,j)}}}}  &  \xrightarrow{\cong} & \mathbb{Z}_{d_{i,j}} \\  y_{\a_{(i,j)}}  &
		 	\mapsto & x_{i,j}^{\frac{d_{i,j}}{p_{\a_{(i,j)}}^{k_{\a_{(i,j)}}}}}, \text{ for all } \a\in \{1,\dots,n\}
		 	
		 \end{array}   $ \vspace{2mm}

		 Using the procedure described in Corollary \ref{Cor1} we can reduce the image of the set $X$ under these isomorphisms to the set $S$, i.e the union of the sets $S_{(i,j)}^{(i',j')}$. Therefore $\gamma_{\om}(G)$ is the normal closure of the subset $S$ in $G$.
\end{proof}

\section{Calculation of $(N_{p})_{\om}(G)$ for GBS tree-groups}
\subsection{The segment}
We will now calculate the intersection  $(N_{p})_{\om}(G)$ of the subgroups of finite $p$-index of $G$, $p$ prime, in the case where the  underlying graph of the GBS group $G$ is a segment. In order to prove the result we will need the following Lemma in which we calculate the abelianization of $G$.
\begin{lemma}\label{abelseg}
	Let $G$ be the group $G=\langle a,b\hspace{1mm}|\hspace{1mm}a^{k}=b^{l}\rangle$, where $d=gcd(k,l)$. 
	\\Then $G^{ab}\cong \Z_{d}\times\Z$.
\end{lemma}
\begin{proof}
We will prove this using Tietze transformations in the presentation of $G^{ab}=\langle a,b\hspace{1mm}|\hspace{1mm}a^{k}=b^{l}, [a,b]=1 \rangle$. By Bezout's identity there are $x,y\in \Z$ such that $xk+yl=d$. Now let $\gamma=a^{\frac{k}{d}}b^{-\frac{l}{d}}$ and $\delta=\a^{y}b^{x}$. Since the determinant of the matrix 
$$\begin{pmatrix}
	\frac{k}{d} & -\frac{l}{d}\\
	y & x
\end{pmatrix}$$ is $1$ we have that $\{\gamma,\delta\}$ is a new generating set for $G^{ab}$ which has now the presentation 
$$G=\langle \gamma,\delta\hspace{1mm}|\hspace{1mm}\gamma^{d}=1,[\gamma,\delta]=1\rangle\cong \Z_{d}\times\Z.$$
Notice that $a=a^{1}=a^{x\frac{k}{d}+y\frac{l}{d}}=(a^{\frac{k}{d}})^{x}(a^{y})^{\frac{l}{d}}=(\gamma\cdot b^{\frac{l}{d}})^{x}(\delta\cdot b^{-x})^{\frac{l}{d}}=\gamma^{x}\cdot \delta^{\frac{l}{d}}$.\\
Simillarly we have that $b=\gamma^{-y}\delta^{\frac{k}{d}}$.
\end{proof}
We can now  calculate the intersection  $(N_{p})_{\om}(G)$, in the case where the  underlying graph of the GBS group $G$ is a segment.
\begin{theorem}\label{segmNp}
		Let $G$ be the group $G=\langle a,b\hspace{1mm}|\hspace{1mm}a^{k}=b^{l}\rangle$, where $ k=p^{\kappa}k_{1}, l=p^{\lambda}l_{1}$ with $p$ prime and $p\nmid k_{1},l_{1}$. Let $d_{1}=gcd(k_{1},l_{1})$. Then  $(N_{p})_{\om}(G)$ is the normal closure of the subset $\{a^{\frac{k}{d_{1}}}b^{-\frac{l}{d_{1}}}, [a,b^{p^{\lambda}}],[a^{p^{\kappa}},b] \}$ in $G$.
\end{theorem}
\begin{proof}
	Let $f$ be a homomorphism from $G$ to a $p$-group $P$. We denote $f(g)$ by $\bar{g}$ for all $g\in G$. We have that 
	$\langle \bar{a} \rangle \geq \langle \bar{a}^{p^{\k}} \rangle = \langle \bar{a}^{k} \rangle = \langle \bar{b}^{l} \rangle = \langle \bar{b}^{p^{\l}} \rangle \leq \langle \bar{b} \rangle$. Therefore the relations $[\bar{a},\bar{b}^{p^{\lambda}}],[\bar{a}^{p^{\kappa}},\bar{b}]$ hold in $P$.

	Let $p^{\mu}=max\{ord(\bar{a}),ord(\bar{b})\}$. Since $gcd(d_{1},p^{\mu})=1$ there are $x,y\in\Z$ such that $x\cdot d_{1}+y\cdot p^{\mu}=1$.    Therefore we have the following equality in $P$.
	$$\bar{a}^{p^{\kappa}\frac{k_{1}}{d_{1}}(x\cdot d_{1})}=\bar{b}^{p^{\lambda}\frac{l_{1}}{d_{1}}(x\cdot d_{1})}\Rightarrow \bar{a}^{p^{\kappa}\frac{k_{1}}{d_{1}}(1-y\cdot p^{\mu})}=\bar{b}^{p^{\lambda}\frac{l_{1}}{d_{1}}(1-y\cdot p^{\mu})}$$
	Now since $\bar{a}^{p^{\mu}}=\bar{b}^{p^{\mu}}=1$ we have that $\bar{a}^{p^{\kappa}\frac{k_{1}}{d_{1}}}=\bar{b}^{p^{\lambda}\frac{l_{1}}{d_{1}}}$ holds in $P$ and thus  $a^{\frac{k}{d_{1}}}b^{-\frac{l}{d_{1}}}\in (N_{p})_{\om}(G)$.

	 Let $S=\{a^{\frac{k}{d_{1}}}b^{-\frac{l}{d_{1}}}, [a,b^{p^{\lambda}}],[a^{p^{\kappa}},b] \}$. Then we have that $S\subseteq (N_{p})_{\om}(G)$. 
	Let now $\tilde{G}=G/S^{G}$, that is $$\tilde{G}= \langle a,b \hspace{1mm}|\hspace{1mm} a^{\frac{k}{d_{1}}}=b^{\frac{l}{d_{1}}},[a,b^{p^{\lambda}}],[a^{p^{\kappa}},b]\rangle$$  
	
	Let $N$ be the normal subgroup of $\tilde{G}$, $N=\langle a^{p^{\k}}b^{-p^{\l}}\rangle\subseteq Z(\tilde{G})$. We have that $\tilde{G}/N=\langle a,b \hspace{1mm}|\hspace{1mm} a^{p^{\k}}=b^{p^{\l}}\rangle$, which by Proposition \ref{prop8} is residually a finite $p$-group. 
	Moreover, by Lemma \ref{abelseg} we have that $\tilde{G}^{ab}=\tilde{G}/\tilde{G}'=\langle a,b \hspace{1mm}|\hspace{1mm} a^{\frac{k}{d_{1}}}=b^{\frac{l}{d_{1}}},[a,b]\rangle$ is also residually a finite $p$-group. 
	
	 Now let $\nu_{1}, \nu_{2}$ be the natural epimorphisms from $\tilde{G}$ to $\tilde{G}/N$ and $\tilde{G}^{ab}$ respectively. If  $g$ is a non trivial element of $\tilde{G}$ with $g\notin N$ then $\nu_{1}(g)$ is not trivial since $Ker\nu_{1}=N$. Now assume that $g\in N$. Using the transformations described in Lemma \ref{abelseg} we can see that $\nu_{2}$ maps the generator $a^{p^{\k}}b^{-p^{\l}}$ of $N$ to a non trivial element of the free abelian factor of $\tilde{G}^{ab}$. Therefore $\nu_{2}(g)$ is not trivial. Consequently, $\tilde{G}$ is residually a finite $p$-group. 
\end{proof}
\subsection{The tree}
We will now calculate $(N_{p})_{\om}(G)$ in the case where the underlying graph of the generalized Baumslag-Solitar group $G$ is a tree $T$. The main idea is the utilization of  Theorem \ref{segmNp}. More specifically, we prove that in order to calculate $(N_{p})_{\om}(G)$ it is sufficient to calculate the intersection $(N_{p})_{\om}$ for each subgroup of $G$ that is defined by a path in the tree. In our proof we will use the following result from Gruenberg.

 \begin{proposition}\label{Gru2}[Gruenberg]\ \\
	If $G$ is a group and $H$ a normal subgroup of $G$ of finite $p$ index,then if $H$ is residually a finite $p$-group then $G$ is is residually a finite $p$-group.
\end{proposition}
\begin{proof}
For proof, see \cite{Gruen}, Lemma 1.5.
\end{proof}
Moreover, we will use the following Proposition which allows us to calculate the intersection $(N_{p})_{\omega}(G)$ of a group $G$ in steps.
\begin{proposition}\label{steps}
	Let $G$ be a group and $H$ be a subgroup of $G$. Then 
	$$\bigslant{G}{(N_{p})_{\omega}(G)} \cong \bigslant{G/(N_{p})_{\omega}(H)}{(N_{p})_{\omega}\big(G/(N_{p})_{\omega}(H)\big)} $$
\end{proposition}
\begin{proof}
	Let $\nu$ be the natural epimorphism $\nu:G \twoheadrightarrow G/(N_{p})_{\omega}(H)$. It is sufficient to see that since $(N_{p})_{\omega}(H)\subseteq (N_{p})_{\omega}(G)$ we have that $\nu\big((N_{p})_{\omega}(G)\big)=(N_{p})_{\omega}\big(G/(N_{p})_{\omega}(H)\big)$ and hence the isomorphism holds.
\end{proof}
Let us now prove the main Theorem.
\begin{theorem}\label{Npw_tree}
	Let $G$ be a GBS tree-group and let
	 $H_{(i,j)}$ be the subgroup of $G$ generated by every pair of distinct vertex groups  $\<x_{i}\>$ and  $\<x_{j}\>$. Let  $S_{(i,j)}$ be the set defined in Theorem \ref{segmNp} for each such subgroup and let $S$ be their union. Then $(N_{p})_{\om}(G)$ is the normal closure of $S$ in $G$.
\end{theorem} 
\begin{proof}
	Since $H_{(i,j)}$ are subgroups of $G$ it is obvious that $S_{(i,j)}$ are contained in $(N_{p})_{\om}(G)$. Using Proposition \ref{steps} we will calculate $(N_{p})_{\om}(G)$ in two steps. 
	
	Consider every pair of generators $x_{i}$ and $x_{j}$ such that $\<x_{i}\>$ and  $\<x_{j}\>$ correspond to adjacent vertices. Then since the relation $x_{i}^{p^{s_{i}}\nu_{i}}=x_{j}^{p^{r_{j}}\mu_{j}}$ holds we know that $$x_{i}^{p^{s_{i}}\frac{\nu_{i}}{d_{i,j}}}\cdot x_{j}^{-p^{r_{j}}\frac{\mu_{i}}{d_{i,j}}}\in (N_{p})_{\om}(G), \text{ where\hspace{2mm}} d_{i,j}=gcd(\nu_{i},\mu_{j})$$
	
	Let $\bar{G}$ be the quotient group of $G$ by the normal closure of the set that consists of the above relations. Then $\bar{G}$ is again a GBS tree-group with different edge groups.  Notice that if $n_{i}=\frac{\nu_{i}}{d_{i,j}}$ and $m_{i}=\frac{\mu_{i}}{d_{i,j}}$ then in $\bar{G}$ if two generators $x_{i}$ and $x_{j}$ correspond to adjacent vertices the following relation holds.
	
	$$x_{i}^{p^{s_{i}}n_{i}}=x_{j}^{p^{r_{j}}m_{j}}, \text{ where\hspace{2mm}} gcd(n_{i},m_{j})=1$$ 
	Consider now every subgroup $\bar{H}_{(i,j)}$ of $\bar{G}$ that is generated by every pair of distinct vertex groups  $\<x_{i}\>$ and  $\<x_{j}\>$ and let  $\bar{S}_{(i,j)}$ be the set defined in Theorem \ref{segmNp} for each such subgroup and  $\bar{S}$ be their union. We shall show that  $\tilde{G}=\bar{G}/\bar{S}^{G}$ is residually a finite $p$-group. 
	
  In order to prove that we will first calculate the abelianization $\tilde{G}^{ab}$ of  $\tilde{G}$ and prove it is residually a finite $p$-group.

 	Let $|V(T)|=N$ and $x_{1},\dots,x_{N}$ be the generators of the vertex groups. Then $\tilde{G}^{ab}$  contains the following type of relations.
 	
 	Relations of the form 
 	
 	$$x_{i}^{p^{s_{i}}n_{i}}=x_{j}^{p^{r_{j}}m_{j}}, \text{ where\hspace{2mm}} gcd(n_{i},m_{j})=1\hspace{5mm} (R_{1})$$ in case where $\<x_{i}\>$ and $\<x_{j}\>$ correspond to adjacent vertices. 
 	
 	If on the other hand we  assume that the vertices that correspond to $\<x_{i}\>$ and  $\<x_{j}\>$ are connected by the following path

 	\begin{center}

 	\vspace{1mm}
 		
 		\begin{tikzpicture}
 			
 			\draw[thick,-] (1,0) -- (4,0) node[midway,sloped,above] {\footnotesize{$p^{s_{i_{1}}}n_{i_{1}}\hspace{7mm}p^{r_{i_{1}}}m_{i_{1}}$}};
 			\draw[thick,-] (4,0) -- (7,0) node[midway,sloped,above] {\footnotesize{$p^{s_{i_{2}}}n_{i_{2}}\hspace{7mm}p^{r_{i_{2}}}m_{i_{2}}$}};
 			
 			\fill (1.0cm,0cm) circle (1.5pt);  \fill (4cm,0cm) circle (1.5pt);
 			\fill (7cm,0cm) circle (1.5pt);
 			
 			\fill (8cm,0cm) circle (1.2pt);
 			\fill (8.5cm,0cm) circle (1.2pt);
 			\fill (9cm,0cm) circle (1.2pt);
 				\fill (10cm,0cm) circle (1.5pt);
 			\fill (13cm,0cm) circle (1.5pt);
 				\draw[thick,-] (10,0) -- (13,0) node[midway,sloped,above] {\footnotesize{$p^{s_{i_{k}}}n_{i_{k}}\hspace{7mm}p^{r_{i_{k}}}m_{i_{k}}$}};

 			\node at (1.0,-0.4) (nodeA) {{$\< x_{i} \> $}};

 			\node at (13,-0.4) (nodeP){{$\< x_{j} \> $}};

 		\end{tikzpicture} 
 	
 	\end{center}
  then we have a relation of the form 
  $$x_{i}^{p^{s_{i}}\frac{n_{i_{1}}\cdots n_{i_{k}}}{d_{i_{1},i_{k}}}}=x_{j}^{p^{r_{j}}\frac{m_{i_{1}}\cdots m_{i_{k}}}{d_{i_{1},i_{k}}}}, \text{where\hspace{2mm}} d_{i_{1},i_{k}}=gcd(n_{i_{1}}\cdots n_{i_{k}},m_{i_{1}}\cdots m_{i_{k}})\hspace{5mm} (R_{2})$$ 
Finally we have the relations that are commutators between the generators $x_{i}$, that is 
  $$[x_{i},x_{j}]=1, \text{ for all } i\neq j \in \{1,\dots,N\}.\hspace{5mm}(R_{3})$$

  For each generator $x_i$, using relations $(R_{1})$ and $(R_{2})$,  we define $M_{i}=max\{s_{i},r_{i}\}$ for all $i\neq j\in \{1,\dots,N\}.$ That is, for each generator $x_i$, we choose the highest power of $p$ that arises in $x_i$ among the relations of $x_i$ with every other generator $x_{j}$, $j\neq i$. Then, if $K$ is the subgroup of $\tilde{G}^{ab}$ generated by $\{ x_{1}^{p^{M_{1}}},  x_{2}^{p^{M_{2}}},\dots, x_{N}^{p^{M_{N}}}\}$, we have that $\tilde{G}^{ab}/K$ is a quotient group of $\Z_{p^{M_{1}}}\times\Z_{p^{M_{2}}}\times\cdots\times\Z_{p^{M_{N}}}$ and thus a finite $p$-group. Therefore, using Gruenberg's Proposition \ref{Gru2} it is sufficient to prove that the subgroup $K$ is residually a finite $p$-group. 
  
  Let $y_{i}=x_{i}^{p^{M_{i}}}$, for all $i\in \{1,\dots,N\}$.
  Now $K$ has a presentation with generating set $\{y_{1},\dots,y_{N}\}$ and  relations of the form:
  \begin{enumerate}
  	\item[($R'_{1})$] $y_{i}^{n_{i,j}}=y_{j}^{m_{i,j}}$
  		\item[($R'_{2})$] $y_{i}^{\frac{n_{i_{1}}\cdots n_{i_{k}}}{d_{i_{1},i_{k}}}}=y_{j}^{\frac{m_{i_{1}}\cdots m_{i_{k}}}{d_{i_{1},i_{k}}}}$
  			\item[($R'_{3})$] $[y_{i},y_{j}]=1$
  \end{enumerate}
\textbf{Claim:} $K\cong \Z$.\\
We will prove the claim  using induction on the number of vertices of the tree.	

Let $|V(T)|=2$. Then we have that $K$ is the abelianization of the fundamental group of the following segment of groups, where $gcd(n,m)=1$.
	
\begin{center}

	\vspace{3mm}
	
	\begin{tikzpicture}
		
		\draw[thick,-] (1,0) -- (4,0) node[midway,sloped,above] {\footnotesize{${n}\hspace{15mm} {m}$}};
	
		\fill (1.0cm,0cm) circle (1.5pt);  \fill (4cm,0cm) circle (1.5pt);
		
		\node at (1.0,-0.4) (nodeA) {{$\< y_{1} \> $}};

		\node at (4,-0.4) (nodeP){{$\< y_{2} \> $}};

	\end{tikzpicture} 
	
\end{center}
Then by Lemma \ref{abelseg} the result holds.

Assume now that $|V(T)|=N$. We will prove that with appropriate changes on the generating set we can take a group that is the abelianization of the fundamental group of a tree of groups where the underlying tree has $N-1$ vertices, in which the hypothesis holds.

 We choose an edge of the underlying tree as depicted in the following figure, where $gcd(m_{1},n_{1})=1$. There are $\k,\l\in \Z$ such that $\k n_{1}+\l m_{1}=1$.
Let $\a=y_{1}^\l y_{2}^{\k}$. Then using Tietze transformation, we can remove $y_{1}$ and $y_{2}$ from the generating set by replacing them with $y_{1}=\a^{{m_{1}}}$ and $y_{2}=\a^{{n_{1}}}$. Therefore, the original underlying tree $T$ is now a tree $T'$ with $N-1$ vertices.

\begin{center}
\vspace{3mm}

	\begin{minipage}{150pt}
		
	\begin{tikzpicture}
		\centering
		\draw[thick,-] (1,0) -- (3,0) node[midway,sloped,above] {\footnotesize{$n_{1}\hspace{10mm} m_{1}$}};
		\draw[dashed,-] (-0.2,0) -- (1,0) node[midway,sloped,above] {\footnotesize{}};
		\draw[dashed,-] (0,1) -- (1,0) node[midway,sloped,above] {\footnotesize{}};
		\draw[dashed,-] (0,-1) -- (1,0) node[midway,sloped,above] {\footnotesize{}};
		\draw[dashed,-] (3,0) -- (4.2,0) node[midway,sloped,above] {\footnotesize{}};
		\draw[dashed,-] (3,0) -- (4,1) node[midway,sloped,above] {\footnotesize{}};
		\draw[dashed,-] (3,0) -- (4,-1) node[midway,sloped,above] {\footnotesize{}};
		\fill (1.0cm,0cm) circle (1.5pt);  \fill (3cm,0cm) circle (1.5pt);
		
		\node at (1.0,-0.4) (nodeA) {{$\< y_{1} \> $}};

		\node at (3,-0.4) (nodeP){{$\< y_{2} \> $}}; 
	\end{tikzpicture} 
	\end{minipage}
\begin{minipage}{80pt}
\begin{tikzpicture}
\hspace*{-0.01cm}
	\draw[ultra thick,->] (1,0) -- (3,0) node[midway,sloped,above] {\footnotesize{Tietze}};
\end{tikzpicture} 
\end{minipage}
	\begin{minipage}{150pt}
		\begin{tikzpicture}
		\centering
		\draw[dashed,-] (1,0) -- (3,0) node[sloped,above] {\footnotesize{}};
		\draw[dashed,-] (1,0) -- (2.8,1) node[sloped,above] {\footnotesize{}};
		\draw[dashed,-] (1,0) -- (2.8,-1) node[sloped,above] {\footnotesize{}};
		\draw[dashed,-] (1,0) -- (-0.8,1) node[sloped,above] {\footnotesize{}};
			\draw[dashed,-] (1,0) -- (-0.8,-1) node[sloped,above] {\footnotesize{}};
				\draw[dashed,-] (1,0) -- (-1,0) node[sloped,above] {\footnotesize{}};
		\fill (1.0cm,0cm) circle (1.5pt); 
		
		\node at (1.0,-0.4) (nodeA) {{$\< \a \> $}};

	\end{tikzpicture} 
	\end{minipage}
\end{center}
\vspace{3mm}
Let us now check the relations of $K$. Let $y_{k}$ be a random generator that corresponds to a vertex in $T$ and assume that the following path connects the vertices $\< y_{1} \>$ and $\<y_{k} \>$. 
	\begin{center}

	\vspace{3mm}
	
	\begin{tikzpicture}
		
		\draw[thick,-] (1,0) -- (4,0) node[midway,sloped,above] {\footnotesize{$n_{1}\hspace{10mm}m_{1}$}};
		\draw[thick,-] (4,0) -- (7,0) node[midway,sloped,above] {\footnotesize{$n_{2}\hspace{10mm}m_{2}$}};
		
		\fill (1.0cm,0cm) circle (1.5pt);  \fill (4cm,0cm) circle (1.5pt);
		\fill (7cm,0cm) circle (1.5pt);
		
		\fill (8cm,0cm) circle (1.2pt);
		\fill (8.5cm,0cm) circle (1.2pt);
		\fill (9cm,0cm) circle (1.2pt);
		\fill (10cm,0cm) circle (1.5pt);
		\fill (13cm,0cm) circle (1.5pt);
	
			\draw[thick,-] (10,0) -- (13,0) node[midway,sloped,above] {\footnotesize$n_{k-1}\hspace{10mm}m_{k-1}$};
	
\node at (1.0,-0.4) (nodeA) {{$\< y_{1} \> $}}; 

\node at (4.0,-0.4) (nodeA) {{$\< y_{2} \> $}}; 

\node at (7.0,-0.4) (nodeA) {{$\< y_{3} \> $}};

\node at (10,-0.4) (nodeA) {{$\< y_{k-1} \> $}}; 
\node at (13,-0.4) (nodeP){{$\< y_{k} \> $}};

	\end{tikzpicture} 
	
\end{center}
The first subpath in which we have the corresponding relations
\begin{itemize}
	\item[] $y_{1}^{{n_{1}}}=y_{2}^{{m_{1}}}$
	\item[] $y_{2}^{{n_{2}}}=y_{3}^{{m_{2}}}$
	\item[] $y_{1}^{\frac{n_{1}n_{2}}{d_{1,2}}}=y_{3}^{\frac{m_{1}m_{2}}{d_{1,2}}}$
\end{itemize}
will be be replaced by the following segment

\begin{center}
	\vspace{3mm}

	\begin{minipage}{180pt}
		
		\begin{tikzpicture}[scale=0.8]
			\centering
			
		\draw[thick,-] (1,0) -- (4,0) node[midway,sloped,above] {\footnotesize{$n_{1}\hspace{10mm}m_{1}$}};
		\draw[thick,-] (4,0) -- (7,0) node[midway,sloped,above] {\footnotesize{$n_{2}\hspace{10mm}m_{2}$}};
		
		\fill (1.0cm,0cm) circle (1.5pt);  \fill (4cm,0cm) circle (1.5pt);
		\fill (7cm,0cm) circle (1.5pt);

		\node at (1.0,-0.4) (nodeA) {{$\< y_{1} \> $}}; 
		
		\node at (4.0,-0.4) (nodeA) {{$\< y_{2} \> $}}; 
		
		\node at (7.0,-0.4) (nodeA) {{$\< y_{3} \> $}};  
		\end{tikzpicture} 
	\end{minipage}
	\begin{minipage}{80pt}
		\begin{tikzpicture}
			\centering
			\draw[ultra thick,->] (5,1) -- (6.5,1) node[midway,sloped,above] {\footnotesize{Tietze}};
		\end{tikzpicture} 
	\end{minipage}
	\begin{minipage}{150pt}
		\begin{tikzpicture}
			\centering
			\draw[thick,-] (1,0) -- (3,0) node[sloped,above] {\footnotesize{}};
			
			\fill (1.0cm,0cm) circle (1.5pt); 
				\fill (3.0cm,0cm) circle (1.5pt); 
			\node at (1.0,-0.4) (nodeA) {{$\< \a \> $}}; 
				\node at (3.0,-0.4) (nodeA) {{$\< y_{3} \> $}};

		\end{tikzpicture} 
	\end{minipage}
\end{center}
where the corresponding relations now will be 
\begin{itemize}
	\item[]$\a ^{{n_{1}n_{2}}}=y_{3}^{{m_{2}}} \hspace{3mm}(1)$
	\item[] $\a^{m_{1}\cdot \frac{n_{1}n_{2}}{d_{1,2}}}=y_{3}^{\frac{m_{1}m_{2}}{d_{1,2}}}\hspace{3mm}(2)$
\end{itemize}
We know that that $d_{1,2}= gcd(n_{1},m_{2})\cdot gcd(n_{2},m_{1})$. 

Let $d_{1}=gcd(n_{1},m_{2})$ and $d_{2}=gcd(n_{2},m_{1})$. Thus 

\[
\begin{array}{lll}
	m_{2}=d_{1}m_{2}'& \text{and} & 	n_{1}=d_{1}n_{1}' \\
		n_{2}=d_{2}n_{2}'& \text{and} & 	m_{1}=d_{2}m_{1}'
\end{array} \] 
Hence $(1)$ now becomes $\a^{d_{1}(n_{1}'n_{2})}=y_{3}^{d_{1}m_{2}'}\hspace{3mm} (1')$. \\ While $(2)$ becomes $\a^\frac{d_{2}m_{1}'d_{2}n_{1}'n_{2}}{d_{1}d_{2}}=y_{3}^{\frac{d_{2}m_{1}'d_{1}m_{2}'}{d_{1}d_{2}}}\Rightarrow \a^{m_{1}'(n_{1}'n_{2})}=y_{3}^{m_{1}'(m_{2}')}\hspace{3mm}(2')$\\

Now since $gcd(d_{1},m_{1})=1$ we have that the relation that finally corresponds to the segment between $\< \a \>$ and 
$\< y_{3} \> $ is the relation $\a^{\nu}=y_{3}^{\mu}$, where $\nu=n_{1}'n_{2}$ and $\mu=m_{2}'$. We notice that since $gcd(m_{2}',n_{2})=1$ and  $gcd(m_{2}',n_{1}')=1$ we also conclude that $gcd(\nu,\mu)=1.$ 

The relations between the generator $y_{k}$ with the generators $ y_{1}$ and $y_{2}$ are the following.

\begin{enumerate}
	\item[]  $y_{1}^{\frac{n_{1}\cdots n_{k-1}}{d_{1,k-1}}}=y_{k}^{\frac{m_{1}\cdots m_{k-1}}{d_{1,k-1}}} \hspace{3mm}(3)$
	\item[]  $y_{2}^{\frac{n_{2}\cdots n_{k-1}}{d_{2,k-1}}}=y_{k}^{\frac{m_{2}\cdots m_{k-1}}{d_{2,k-1}}} \hspace{3mm}(4)$

\end{enumerate}
Using the above Tietze transformations, these relations will become 
\begin{enumerate}
	\item[]  $\a^{\frac{m_{1}\cdot n_{1}\cdots n_{k-1}}{d_{1}\cdot d_{1,k-1}}}=y_{k}^{\frac{m_{1}\cdots m_{k-1}}{d_{1,k-1}}} \hspace{3mm}(3')$
	\item[]  $\a^{ \frac{n_{1}\cdot n_{2}\cdots n_{k-1}}{d_{1}d_{2,k-1}}}=y_{k}^{\frac{m_{2}\cdots m_{k-1}}{d_{2,k-1}}} \hspace{3mm}(4')$
	
\end{enumerate}
We will prove that $(3'),(4')$ give us the following relation 
$$\a^{\frac{\nu n_{3}\cdots n_{k-1}}{d}}=y_{k}^{\frac{\mu m_{3}\cdots m_{k-1}}{d}}$$
where $d=gcd(\nu n_{3}\cdots n_{k-1},\mu m_{3}\cdots m_{k-1})=gcd(\frac{n_{1}}{d_{1}}n_{2}n_{3}\cdots n_{k-1},\frac{m_{2}}{d_{1}}m_{3}\cdots m_{k-1})$.

Let $\d=gcd(n_{1}n_{2}\cdots n_{k-1},m_{2}m_{3}\cdots m_{k-1}).$ Then $\d=d_{1}d$. Moreover since $\d|d_{1,k-1}$ we have that $ d_{1,k-1}=\d \cdot \d_{1}$ for some integer $\d_{1}$, while since $d_{2,k-1}|\d$ we have that $\d=d_{2,k-1}\cdot \d_{2}
$ for some integer $\d_{2}$. 

Now the relation $(3')$ can be written as $\a ^{(\frac{n_{1}\cdots n_{k-1}}{\d})\frac{m_{1}}{\d_{1}}}= y_{k} ^{(\frac{m_{2}\cdots m_{k-1}}{\d})\frac{m_{1}}{\d_{1}}}$, while the relation $(4')$ can be written as $\a ^{(\frac{n_{1}\cdots n_{k-1}}{\d})\d_{2}}= y_{k} ^{(\frac{m_{2}\cdots m_{k-1}}{\d})\d_{2}}$. But since $gcd(\frac{m_{1}}{\d_{1}},\d_{2})=1$ we have that  $\a ^{(\frac{n_{1}\cdots n_{k-1}}{\d})}= y_{k} ^{(\frac{m_{2}\cdots m_{k-1}}{\d})}\Leftrightarrow \a ^{(\frac{n_{1}\cdots n_{k-1}}{d_{1}d})}= y_{k} ^{(\frac{m_{2}\cdots m_{k-1}}{d_{1}d})}$. Finally, since $n_{1}n_{2}=d_{1}\nu$ and $m_{2}=d_{1}\mu$, we have $\a^{\frac{\nu n_{3}\cdots n_{k-1}}{d}}=y_{k}^{\frac{\mu m_{3}\cdots m_{k-1}}{d}}$.
\vspace{2mm}

If on the other hand the path that connects the vertices $\< y_{1} \>$ and $\<y_{k} \>$ does not contain the vertex $\< y_{2} \>$ it suffices to interchange the roles of $\< y_{1} \>$ and $\<y_{2} \>$ in the above procedure. Therefore we proved our claim.

Consider now the abelian subgroup $K$ generated by $\{ x_{1}^{p^{M_{1}}},  x_{2}^{p^{M_{2}}},\dots, x_{N}^{p^{M_{N}}}\}$ as a subgroup of $\tilde{G}$ and let $\hat{G}=\tilde{G}/{K}$. We will prove that $\hat{G}$ is also residually a finite $p$-group. First, since the relation  $x_{i}^{p^{M_{i}}}=1$ holds for every $i\in\{1,\dots,N\}$ we can remove all the commutators from the presentation of $\hat{G}$.

 Moreover we claim that every relation between generators that correspond to non-adjacent vertices arises from relations between adjacent vertices and therefore these relations can also be removed from the presentation. Indeed if
that vertices that correspond to $\<x_{i}\>$ and  $\<x_{j}\>$ are connected by the following path

\begin{center}

	\vspace{1mm}
	
	\begin{tikzpicture}
		
		\draw[thick,-] (1,0) -- (4,0) node[midway,sloped,above] {\footnotesize{$p^{s_{i_{1}}}n_{i_{1}}\hspace{7mm}p^{r_{i_{1}}}m_{i_{1}}$}};
		\draw[thick,-] (4,0) -- (7,0) node[midway,sloped,above] {\footnotesize{$p^{s_{i_{2}}}n_{i_{2}}\hspace{7mm}p^{r_{i_{2}}}m_{i_{2}}$}};
		
		\fill (1.0cm,0cm) circle (1.5pt);  \fill (4cm,0cm) circle (1.5pt);
		\fill (7cm,0cm) circle (1.5pt);
		
		\fill (8cm,0cm) circle (1.2pt);
		\fill (8.5cm,0cm) circle (1.2pt);
		\fill (9cm,0cm) circle (1.2pt);
		\fill (10cm,0cm) circle (1.5pt);
		\fill (13cm,0cm) circle (1.5pt);
		\draw[thick,-] (10,0) -- (13,0) node[midway,sloped,above] {\footnotesize{$p^{s_{i_{k}}}n_{i_{k}}\hspace{7mm}p^{r_{i_{k}}}m_{i_{k}}$}};

		\node at (1.0,-0.4) (nodeA) {{$\< x_{i} \> $}};

		\node at (13,-0.4) (nodeP){{$\< x_{j} \> $}};

	\end{tikzpicture} 
	
\end{center}
then using the relations between the adjacent vertices we have the following relation between $x_{i}$ and $x_{j}$.
  $$x_{i}^{p^{s_{i}}{n_{i_{1}}\cdots n_{i_{k}}}}=x_{j}^{p^{r_{j}}{m_{i_{1}}\cdots m_{i_{k}}}}$$
Let $M_{i,j}=max\{M_{i},M_{j}\}$. If $d_{i_{1},i_{k}}=gcd(n_{i_{1}}\cdots n_{i_{k}},m_{i_{1}}\cdots m_{i_{k}})$, then $gcd(d_{i_{1},i_{k}},p^{M_{i,j}})=1$ and thus there are $z,w\in\Z$ such that $d_{i_{1},i_{k}}\cdot z+p^{M_{i,j}}\cdot w=1$. Therefore, raising the above relation to $z$ we have
 $$x_{i}^{p^{s_{i}}\frac{n_{i_{1}}\cdots n_{i_{k}}}{d_{i_{1},i_{k}}}(d_{i_{1},i_{k}}\cdot z)}=x_{j}^{p^{r_{j}}\frac{m_{i_{1}}\cdots m_{i_{k}}}{d_{i_{1},i_{k}}}(d_{i_{1},i_{k}}\cdot z)}\Leftrightarrow$$
  $$x_{i}^{p^{s_{i}}\frac{n_{i_{1}}\cdots n_{i_{k}}}{d_{i_{1},i_{k}}}(1-p^{M_{i,j}}\cdot w)}=x_{j}^{p^{r_{j}}\frac{m_{i_{1}}\cdots m_{i_{k}}}{d_{i_{1},i_{k}}}(1-p^{M_{i,j}}\cdot w)}$$
But since $x_{i}^{p^{M_{i,j}}}=x_{j}^{p^{M_{i,j}}}=1$ we get $x_{i}^{p^{s_{i}}\frac{n_{i_{1}}\cdots n_{i_{k}}}{d_{i_{1},i_{k}}}}=x_{j}^{p^{r_{j}}\frac{m_{i_{1}}\cdots m_{i_{k}}}{d_{i_{1},i_{k}}}}$ and so our claim holds. Notice now that $\hat{G}$ is a fundamental group of a tree $T$ of finite cyclic $p$-groups. Therefore, by Proposition \ref{ptree} we have that $\hat{G}$ is indeed residually a finite $p$-group.

Now let $1\neq g\in \tilde{G}$ and  let $\nu_{1}, \nu_{2}$ be the natural epimorphisms from $\tilde{G}$ to $\hat{G}$ and $\tilde{G}^{ab}$ respectively. If $g\notin {K}=Ker\nu_{1}$ then $\nu_{1}(g)$ is not trivial. On the other hand, if $g\in {K}$, then since $\nu_{2}(K)=K$ which is isomorphic to $\Z$, and so residually a finite $p$-group, we have that $\nu_{2}(g)$ is not trivial. Therefore we conclude that $\tilde{G}$ is residually a finite $p$-group.
 \end{proof}

\section{An example}
We now provide an example where we calculate the intersections $\g_{\omega}(G)$ and $(N_{p})_{\om}(G)$, where $G$ is a GBS tree group. 
	
	Let $G$ be the GBS group presented by the following labeled tree. 
	\begin{center}

		\begin{tikzpicture}
			
			\draw[thick,-] (1.1,0) -- (4,0) node[midway,sloped,above] {\footnotesize{$42\hspace{12mm}30$}};
			
			\draw[ thick,-] (4,0) -- (7,1.7) node[midway,sloped,above] {\footnotesize{$14\hspace{15mm}3$}};
			\draw[ thick,-] (4,0) -- (7,-1.7) node[midway,sloped,above] {\footnotesize{$21\hspace{14mm}12$}};

			\draw[thick,-] (7,1.7) -- (10,1.7)node[midway,sloped,above] {\footnotesize{$10\hspace{12mm}15$}};

			\fill (1.1cm,0cm) circle (1.5pt);  \fill (4cm,0cm) circle (1.5pt);
			\fill (7cm,1.7cm) circle (1.5pt);\fill (7cm,-1.7cm) circle (1.5pt);
			\fill (10cm,1.7cm) circle (1.5pt);

			\node at (1.1,-0.4) (nodeA) {{$\< \a \> $}}; 
			\node at (4,-0.4) (nodeB) {{$\< \b \> $}}; 
			
			\node at (7.3,-1.9) (nodeP) {{$\< \d \> $}}; 
			\node at (7.2,1.4) (nodeP) {{$\< \g \> $}};

			\node at (10.4,1.5) (nodeP){{$\< \e \> $}};

		\end{tikzpicture} 
	\end{center}
	The group $G$ has the following presentation. 
	
	$$G=\< \a,\b,\g,\d,\e \hspace{1mm}|\hspace{1mm}\a^{42} =\b^{30}, \b^{21} =\d^{12}, \b^{14} =\g^3, \g^{10} = \e^{15} \> $$
	We calculate $\gamma_{\omega}(G)$ and $(N_{3})_{\om}(G)$.

 \subsection{Calculation of $\gamma_{\omega}(G)$}
	Using the relations of $G$, the center $Z(G)$ of $G$ is the cyclic subgroup $$Z(G)= \< \a^{588} \> =  \< \b^{420} \> =  \< \g^{90} \> =   \< \d^{240} \> =  \< \e^{135} \>. $$
	
	Let $\bar{G}$ be the group $G/Z(G)$.  Using the following isomorphisms \\
	\vspace{3mm}
	$  \begin{array}{rcl}
		f_{1}: \mathbb{Z}_{588} &  \xrightarrow{\cong} & \mathbb{Z}_{2^2}\times\mathbb{Z}_{3}\times\mathbb{Z}_{7^2} \\ \alpha &
		\mapsto & x_{1}^{-5}x_{2}^{25}x_{3}^{3} 
		
	\end{array}   $

	$  \begin{array}{rcl}
		f_{2}: \mathbb{Z}_{420} &  \xrightarrow{\cong} & \mathbb{Z}_{2^2}\times\mathbb{Z}_{3}\times\mathbb{Z}_{5}\times\mathbb{Z}_{7} \\ \b &
		\mapsto & y_{1}^{85}y_{2}^{-85}y_{3}^{34}y_{4}^{2}
		
	\end{array}   $

	$  \begin{array}{rcl}
		f_{3}: \mathbb{Z}_{90} &  \xrightarrow{\cong} & \mathbb{Z}_{2}\times\mathbb{Z}_{3^2}\times\mathbb{Z}_{5} \\ \g &
		\mapsto & w_{1}^{-7}w_{2}^{28}w_{3}^{2} 
		
	\end{array}   $

	$  \begin{array}{rcl}
		f_{4}: \mathbb{Z}_{240} &  \xrightarrow{\cong} & \mathbb{Z}_{2^4}\times\mathbb{Z}_{3}\times\mathbb{Z}_{5} \\ \d &
		\mapsto & z_{1}^{95}z_{2}^{-19}z_{3}^{2} 
		
	\end{array}   $

	$  \begin{array}{rcl}
		f_{5}: \mathbb{Z}_{135} &  \xrightarrow{\cong} & \mathbb{Z}_{3^3}\times\mathbb{Z}_{5} \\ \e &
		\mapsto & t_{1}^{11}t_{2}^{-2} 
		
	\end{array}   $

	we have that $\bar{G}$ is isomorphic to  $\hat{G}$ with presentation $$\hat{G}=\< S_{\hat{G}} \hspace{1mm}|\hspace{1mm} R_{ord}, R_{com},R_{amalg} \>$$ where
	\begin{itemize} 
		\item[] $S_{\hat{G}}$ is the generating set 
		$$\{ x_{1},x_{2},x_{3},y_{1},y_{2},y_{3},y_{4},z_{1},z_{2},z_{3},w_{1},w_{2},w_{3},t_{1},t_{2}\}$$ 
		\item[] ${R}_{ord}$ are the relators
		$$\{ x_{1}^{4},,x_{2}^3,x_{3}^{49},y_{1}^4,y_{2}^3,y_{3}^5,y_{4}^7,z_{1}^{16},z_{2}^3,z_{3}^5,w_{1}^2,w_{2}^{9},w_{3}^5,t_{1}^{27},t_{2}^5\}$$ 
		\item[] ${R}_{com}$ are the relators

		\begin{center}
			$\begin{Bmatrix}
				[x_{1},x_{2}],[x_{1},x_{3}],[x_{2},x_{3}],[y_{1},y_{2}],[y_{1},y_{3}],[y_{1},y_{4}],[y_{2},y_{3}],[y_{2},y_{4}],\\
				[y_{3},y_{4}], [z_{1},z_{2}],[z_{1},z_{3}],[z_{2},z_{3}],[w_{1},w_{2}],[w_{1},w_{3}],[w_{2},w_{3}]
			\end{Bmatrix}$
		\end{center}
		
		\item[] and  ${R}_{amalg}$ are the relations
		\begin{center}
			$\begin{Bmatrix}
				(x_{1}^{-5}x_{2}^{25}x_{3}^{3} )^{42} =(y_{1}^{85}y_{2}^{-85}y_{3}^{34}y_{4}^{2})^{30}, (y_{1}^{85}y_{2}^{-85}y_{3}^{34}y_{4}^{2})^{21} =(z_{1}^{95}z_{2}^{-19}z_{3}^{2} )^{12},\\
				(y_{1}^{85}y_{2}^{-85}y_{3}^{34}y_{4}^{2})^{14} =(w_{1}^{-7}w_{2}^{28}w_{3}^{2} 
				)^3, (w_{1}^{-7}w_{2}^{28}w_{3}^{2} 
				)^{10} = (t_{1}^{11}t_{2}^{-2})^{15}
			\end{Bmatrix}$
		\end{center}
		Using the relations $R_{ord}$ and $R_{com}$ the above are equivalent to the relations $$\{x_{1}^{2}x_{3}^{28}=y_{1}^{2}y_{4}^{4},\hspace{1mm} y_{1}y_{3}^{4}=z_{1}^{4}z_{3}^{4},\hspace{1mm} y_{1}^{2}y_{2}y_{3}=w_{1}w_{2}^{6}w_{3},\hspace{1mm} w_{2}=t_{1}^{3}\}$$
	\end{itemize} 
	Now, by Corollary \ref{coprime} the following set $S$ is contained in $\gamma_{\om}(\hat{G})$. \vspace*{-5mm}
	\begin{center}
		$\begin{Bmatrix}
			[x_{1},y_{2}],[x_{1},y_{3}],[x_{1},y_{4}],[x_{1},z_{2}],[x_{1},z_{3}],[x_{1},w_{2}],[x_{1},w_{3}],[x_{1},t_{1}],[x_{1},t_{2}],[x_{2},y_{1}],[x_{2},y_{3}],\\
			[x_{2},y_{4}],[x_{2},z_{1}],[x_{2},z_{3}], [x_{2},w_{1}],[x_{2},w_{3}],[x_{2},t_{2}],[x_{3},y_{1}],[x_{3},y_{2}],[x_{3},y_{3}],[x_{3},z_{1}],[x_{3},z_{2}],\\
			[x_{3},z_{3}], [x_{3},w_{1}],[x_{3},w_{2}],[x_{3},w_{3}],[x_{3},t_{1}],[x_{3},t_{2}],[y_{1},z_{2}],[y_{1},z_{3}],[y_{1},w_{2}],[y_{1},w_{3}], [y_{1},t_{1}],	\\
			[y_{1},t_{2}],[y_{2},z_{1}],[y_{2},z_{3}],[y_{2},w_{1}],[y_{2},w_{3}],[y_{2},t_{2}],[y_{3},z_{1}],[y_{3},z_{2}],[y_{3},w_{1}],[y_{3},w_{2}],[y_{3},t_{1}],\\
			[y_{4},z_{1}],[y_{4},z_{2}],[y_{4},z_{3}],[y_{4},w_{1}],[y_{4},w_{2}],[y_{4},w_{3}], [y_{4},t_{1}],[y_{4},t_{2}],[z_{1},w_{2}],[z_{1},w_{3}],[z_{1},t_{1}],\\
			[z_{1},t_{2}],[z_{2},w_{1}],[z_{2},w_{3}],[z_{2},t_{2}], [z_{3},w_{1}],[z_{3},w_{2}],[z_{3},t_{1}],[w_{1},t_{1}],[w_{1},t_{2}],[w_{2},t_{2}],[w_{3},t_{1}]
		\end{Bmatrix}$
	\end{center}

	We prove that the normal closure of $S$ is $\gamma_{\om}(\hat{G})$ by proving that the group $H=\hat{G}/S^{\hat{G}}$ is residually nilpotent. Notice that the set $S$ contains all the commutators between the generators of $\hat{G}$ with coprime orders (some of which may be trivial).

	Let $1\neq g\in H$. Then $g$ can be written as a word $g=\chi_{1}\chi_{2}\dots\chi_{r}$, where each word $\chi_{i}$  is of the form $x_{1}^{\k_{1,i}}x_{2}^{\k_{2,i}}x_{3}^{\k_{3,i}}y_{1}^{\l_{1,i}}y_{2}^{\l_{2,i}}y_{3}^{\l_{3,i}}y_{4}^{\l_{4,i}}z_{1}^{\mu_{1,i}}z_{2}^{\mu_{2,i}}z_{3}^{\mu_{3,i}}w_{1}^{\nu_{1,i}}w_{2}^{\nu_{2,i}}w_{3}^{\nu_{3,i}}t_{1}^{\rho_{1,i}}t_{2}^{\rho_{2,i}}$, for all $i\in \{1,\dots,r\}$.
	Using the relations of $H$, each subword  $\chi_{i}$ can be rewritten as 
	${w}_{i,2}{w}_{i,3}{w}_{i,5}{w}_{i,7}$, where each word ${w}_{i,p}$ is a word that consists only of  generators with orders powers of the prime number $p$. Now $g={w}_{1,2}{w}_{1,3}{w}_{1,5}{w}_{1,7}\dots{w}_{r,2}{w}_{r,3}{w}_{r,5}{w}_{r,7}$.
	Again using the relations of $H$, we have that $[w_{j_{1},p_{1}},w_{j_{2},p_{2}}]=1$ for all  $j_{1}, j_{2}\in \{1,\dots,r\}$ and for all prime numbers $p_{1}\neq p_{2}\in \{2,3,5,7\}$. Hence $g$ can be written as the word $$g={w}_{1,2}{w}_{2,2}\dots {w}_{r,2}{w}_{1,3}{w}_{2,3}\dots {w}_{r,3}{w}_{1,5}{w}_{2,5}\dots {w}_{r,5}{w}_{1,7}{w}_{2,7}\dots {w}_{r,7}.$$
	Finally, let $W_{p}$ be the word ${w}_{1,p}{w}_{2,p}\dots {w}_{r,p}$ for all $p\in \{2,3,5,7\}$,  and thus $g={W}_{2}{W}_{3}{W}_{5}{W}_{7}$.
	Notice that each word ${W}_{p}$ is a word that consists only of  generators with order powers of the prime number $p$. Now since $g\neq 1 $ at least one of the words ${W}_{2}, {W}_{3}, {W}_{5}$ and ${W}_{7}$ is not trivial. 
	
	If  $W_{2}\neq 1$, let $G_{2}$ be the group 
	$$G_{2}=\< \bar{x}_{1}, \bar{y}_{1}, \bar{z}_{1}\bar{w}_{1}\hspace{1mm}|\hspace{1mm}\bar{x}_{1}^{4}=\bar{y}_{1}^{4}=\bar{z}_{1}^{16}=\bar{w}_{1}^{2}=1, \bar{x}_{1}^{2}=\bar{y}_{1}^{2}, \bar{y}_{1}=\bar{z}_{1}^{4}, \bar{y}_{1}^{2}=\bar{w}_{1}\>$$
	That is the fundamental group of the following graph of groups.
	\begin{center}
		\vspace{1mm}
		
		\begin{tikzpicture}
			
			\draw[thick,-] (1.1,0) -- (4,0) node[midway,sloped,above] {\footnotesize{$2\hspace{12mm}2$}};
			
			\draw[ thick,-] (4,0) -- (7,1.7) node[midway,sloped,above] {\footnotesize{$2\hspace{15mm}1$}};
			\draw[ thick,-] (4,0) -- (7,-1.7) node[midway,sloped,above] {\footnotesize{$1\hspace{14mm}4$}};

			\fill (1.1cm,0cm) circle (1.5pt);  \fill (4cm,0cm) circle (1.5pt);
			\fill (7cm,1.7cm) circle (1.5pt);\fill (7cm,-1.7cm) circle (1.5pt);

			\node at (1,-0.4) (nodeA) {{$\< \bar{x}_{1} \>\cong \mathbb{Z}_{4} $}}; 
			\node at (3.5,-0.4) (nodeB) {{$\< \bar{y}_{1} \>\cong \mathbb{Z}_{4} $}};
			
			\node at (8,-1.9) (nodeP) {{$\< \bar{z}_{1} \>\cong \mathbb{Z}_{16} $}};
			\node at (7.9,1.4) (nodeP) {{$\< \bar{w}_{1} \>\cong \mathbb{Z}_{2} $}};

		\end{tikzpicture} 
	\end{center}
	
	Let $\f_{2}$ be the homomorphism $\f_{2} : H \to G_{2}$ that
	maps $x_{1}$ to $\bar{x}_{1}$, $y_{1}$ to $\bar{y}_{1}$, $z_{1}$ to $\bar{z}_{1}$, $w_{1}$ to $\bar{w}_{1}$ and the remaining generators of $H$ to $1$. (The map $\f_{2}$ is indeed a homomorphism since it preserves the relations of $H$). Since $W_{2}\neq1 $ we have that $\phi_{2}(W_{2})\neq 1$ and hence $\phi_{2}(g)\neq 1$. Since the group $G_{2}$ is residually nilpotent (by Proposition \ref{ptree}) we have that $H$ is also residually nilpotent.

	If  $W_{3}\neq 1$, let $G_{3}$ be the group 
	$$G_{3}=\< \bar{x}_{2}, \bar{y}_{2}, \bar{z}_{2}\bar{w}_{2},\bar{t}_{1}\hspace{1mm}|\hspace{1mm}\bar{x}_{2}^{3}=\bar{y}_{2}^{3}=\bar{z}_{2}^{3}=\bar{w}_{2}^{9}=\bar{t}_{1}^{27}=1, \bar{y}_{2}=\bar{w}_{2}^{6}, \bar{w}_{2}=\bar{t}_{1}^{3}\>$$
	That is the group $\<\bar{x}_{2}\> * \<\bar{z}_{2}\> * \bar{G}_{3} $, where $\<\bar{x}_{2}\>\cong \mathbb{Z}_{3}$, $\<\bar{z}_{2}\>\cong \mathbb{Z}_{3}$ and $\bar{G}_{3}$ is fundamental group of the following graph of groups.
	\begin{center}
		\vspace{1mm}
		
		\begin{tikzpicture}
			
			\draw[thick,-] (1.1,0) -- (4,0) node[midway,sloped,above] {\footnotesize{$1\hspace{16mm}6$}};
			
			\draw[ thick,-] (4,0) -- (7,0) node[midway,sloped,above] {\footnotesize{$1\hspace{16mm}3$}};
			
			\fill (1.1cm,0cm) circle (1.5pt);  \fill (4cm,0cm) circle (1.5pt); \fill (7cm,0cm) circle (1.5pt);

			\node at (1,-0.4) (nodeA) {{$\< \bar{y}_{2} \>\cong \mathbb{Z}_{3} $}}; 
			\node at (3.9,-0.4) (nodeB) {{$\< \bar{w}_{2} \>\cong \mathbb{Z}_{9} $}};
			
			\node at (6.8,-0.4) (nodeP) {{$\< \bar{t}_{1} \>\cong \mathbb{Z}_{27} $}};
			
		\end{tikzpicture} 
	\end{center}
	Notice that $\bar{G}_{3}$ is isomorphic to the group  $\<\bar{t}_{1}\>\cong \mathbb{Z}_{27}$, since we can remove the generators $\bar{y}_{2}$ and $\bar{w}_{2}$ from the presentation, using Tietze transformations.
	
	Let $\f_{3}$ be the homomorphism $\f_{3} : H \to G_{3}$ that
	maps $x_{2}$ to $\bar{x}_{2}$, $y_{2}$ to $\bar{y}_{2}$, $z_{2}$ to $\bar{z}_{2}$, $w_{2}$ to $\bar{w}_{2}$,  $t_{1}$ to $\bar{t}_{1}$ and the remaining generators of $H$ to $1$. Since $W_{3}\neq1 $ we have that $\phi_{3}(W_{2})\neq 1$ and hence $\phi_{3}(g)\neq 1$. Now since the group $\bar{G}_{3}$ is residually a finite $3$-group by Proposition \ref{ptree} and therefore the free product $\<\bar{x}_{2}\> * \<\bar{z}_{2}\> * \bar{G}_{3} $ is residually a finite $3$-group and thus residually nilpotent. Hence, the group $H$ is also residually nilpotent.

	If  $W_{5}\neq 1$, let $G_{5}$ be the group 
	$$G_{5}=\<  \bar{y}_{3}, \bar{z}_{3}\bar{w}_{3},\bar{t}_{2}\hspace{1mm}|\hspace{1mm}\bar{y}_{3}^{5}=\bar{z}_{3}^{5}=\bar{w}_{3}^{5}=\bar{t}_{2}^{5}=1, \bar{y}_{3}^{4}=\bar{z}_{3}^{4}, \bar{y}_{3}=\bar{t}_{3}\>$$
	That is the group $\<\bar{t}_{2}\>  * \bar{G}_{5} $, where $\<\bar{t}_{2}\>\cong \mathbb{Z}_{5}$, $\<\bar{z}_{2}\>\cong \mathbb{Z}_{3}$ and $\bar{G}_{3}$ is fundamental group of the following graph of groups.
	\begin{center}

		\begin{tikzpicture}
			
			\draw[thick,-] (1.1,0) -- (4,0) node[midway,sloped,above] {\footnotesize{$4\hspace{16mm}4$}};
			
			\draw[ thick,-] (4,0) -- (7,0) node[midway,sloped,above] {\footnotesize{$1\hspace{16mm}1$}};
			
			\fill (1.1cm,0cm) circle (1.5pt);  \fill (4cm,0cm) circle (1.5pt); \fill (7cm,0cm) circle (1.5pt);

			\node at (1,-0.4) (nodeA) {{$\< \bar{z}_{3} \>\cong \mathbb{Z}_{5} $}}; 
			\node at (3.9,-0.4) (nodeB) {{$\< \bar{y}_{3} \>\cong \mathbb{Z}_{5} $}};
			
			\node at (6.8,-0.4) (nodeP) {{$\< \bar{w}_{3} \>\cong \mathbb{Z}_{5} $}};
			
		\end{tikzpicture} 
	\end{center}
	Again, using Tietze transformations we can remove the generators $\bar{y}_{3}$ and $\bar{w}_{3}$ and thus the group $\bar{G}_{5}$ is isomorphic to the group  $\<\bar{z}_{3}\>\cong \mathbb{Z}_{5}$.

	Let $\f_{5}$ be the homomorphism $\f_{5} : H \to G_{5}$ that
	maps  $y_{3}$ to $\bar{y}_{3}$, $z_{3}$ to $\bar{z}_{3}$, $w_{3}$ to $\bar{w}_{3}$,  $t_{2}$ to $\bar{t}_{2}$ and the rest of the generators of $H$ to $1_{G_{5}}$. Since $W_{5}\neq1 $ we have that $\phi_{3}(W_{5})\neq 1$ and hence $\phi_{5}(g)\neq 1$. Now since the groups $\<\bar{t}_{2}\>$ and $\bar{G}_{5}$ are residually finite $5$-groups, their free product $\<\bar{t}_{2}\> *  \bar{G}_{5}$ is also residually a finite $5$-group and thus residually nilpotent. Hence, the group $H$ is also residually nilpotent.

	Finally, if  $W_{5}\neq 1$, let $G_{5}$ be the group 
	$$G_{7}=\<  \bar{x}_{3}, \bar{y}_{4}\hspace{1mm}|\hspace{1mm}\bar{x}_{3}^{49}=\bar{y}_{4}^{7}=1, \bar{x}_{3}^{28}=\bar{y}_{4}^{4}\>$$
	That is the fundamental group of the following graph of groups.
	\begin{center}

		\begin{tikzpicture}
			
			\draw[thick,-] (1.1,0) -- (4,0) node[midway,sloped,above] {\footnotesize{$28\hspace{16mm}4$}};

			\fill (1.1cm,0cm) circle (1.5pt);  \fill (4cm,0cm) circle (1.5pt);
			
			\node at (1,-0.4) (nodeA) {{$\< \bar{x}_{3} \>\cong \mathbb{Z}_{49} $}}; 
			\node at (3.9,-0.4) (nodeB) {{$\< \bar{y}_{4} \>\cong \mathbb{Z}_{7} $}};
			
		\end{tikzpicture} 
	\end{center}
	Let $\f_{7}$ be the homomorphism $\f_{7} : H \to G_{7}$ that
	maps  $x_{3}$ to $\bar{x}_{3}$, $y_{4}$ to $\bar{y}_{4}$ and the rest of the generators of $H$ to $1_{G_{5}}$. Since $W_{7}\neq1 $ we have that  $\phi_{7}(g)\neq 1$. Now since the group $G_{7}$ is residually nilpotent (by Proposition \ref{ptree}) we have that $H$ is also residually nilpotent.

	Therefore,  $\gamma_{\om}(\hat{G})=S^{\hat{G}}$. Now using the inverse of the isomorphisms $f_{1},f_{2},f_{3},f_{4}$ and $f_{5}$ we will calculate $\gamma_{\om}(\bar{G})$ which is equal to  $\gamma_{\om}({G})$.

	$  \begin{array}{rcl}
		f_{1}^{-1}: \mathbb{Z}_{2^2}\times\mathbb{Z}_{3}\times\mathbb{Z}_{7^2} &  \xrightarrow{\cong} &\mathbb{Z}_{588}  \\ x_{1} &
		\mapsto & \a^{3\cdot 7^2} \\  x_{2} &
		\mapsto & \a^{2^{2}\cdot 7^2} \\  x_{3} &
		\mapsto & \a^{2^{2}\cdot 3}
		
	\end{array}   $

	$  \begin{array}{rcl}
		f_{2}^{-1}:  \mathbb{Z}_{2^2}\times\mathbb{Z}_{3}\times\mathbb{Z}_{5}\times\mathbb{Z}_{7}&  \xrightarrow{\cong} &  \mathbb{Z}_{420}  \\ y_{1} &
		\mapsto & \b^{3\cdot5\cdot7} \\  y_{2} &
		\mapsto & \b^{2^{2}\cdot5\cdot7} \\  y_{3} &
		\mapsto & \b^{2^{2}\cdot3\cdot7}\\  y_{4} &
		\mapsto & \b^{2^{2}\cdot3\cdot5}
		
	\end{array}   $

	$  \begin{array}{rcl}
		f_{3}^{-1}: \mathbb{Z}_{2}\times\mathbb{Z}_{3^2}\times\mathbb{Z}_{5}   &  \xrightarrow{\cong} &\mathbb{Z}_{90} \\ w_{1} &
		\mapsto & \g^{3^{2}\cdot 5} \\  w_{2} &
		\mapsto & \g^{2\cdot 5} \\  w_{3} &
		\mapsto & \g^{2\cdot 3^{2}}
		
	\end{array}   $

	$  \begin{array}{rcl}
		f_{4}^{-1}: \mathbb{Z}_{2^4}\times\mathbb{Z}_{3}\times\mathbb{Z}_{5} &  \xrightarrow{\cong} & \mathbb{Z}_{240}  \\ z_{1} &
		\mapsto & \d^{3\cdot 5} \\  z_{2} &
		\mapsto & \d^{2^{4}\cdot 5} \\  z_{3} &
		\mapsto & \d^{2^{4}\cdot 3}
		
	\end{array}   $

	$  \begin{array}{rcl}
		f_{5}^{-1}: \mathbb{Z}_{3^3}\times\mathbb{Z}_{5} &  \xrightarrow{\cong} &\mathbb{Z}_{135}  \\ t_{1} &
		\mapsto & \e^5 \\  t_{2} &
		\mapsto & \e^{3^3}
		
	\end{array}   $

	Consider the following subset $S_{x}$ of $S$.

	\begin{center}
		$\begin{Bmatrix}
			[x_{1},y_{2}],[x_{1},y_{3}],[x_{1},y_{4}],[x_{1},z_{2}],[x_{1},z_{3}],[x_{1},w_{2}],[x_{1},w_{3}],\\ [x_{1},t_{1}],[x_{1},t_{2}],[x_{2},y_{1}],[x_{2},y_{3}], [x_{2},y_{4}],[x_{2},z_{1}],[x_{2},z_{3}],\\ [x_{2},w_{1}],[x_{2},w_{3}],[x_{2},t_{2}],[x_{3},y_{1}],[x_{3},y_{2}],[x_{3},y_{3}],[x_{3},z_{1}],\\
			[x_{3},z_{2}],[x_{3},z_{3}], [x_{3},w_{1}],[x_{3},w_{2}],[x_{3},w_{3}],[x_{3},t_{1}],[x_{3},t_{2}]
		\end{Bmatrix}$
	\end{center} 
	Therefore the following set $S_{1}$ belongs to $\gamma_{\om}(G)$.
	\vspace*{-3mm}
	\begin{center}
		$\begin{Bmatrix}
			[\a^{3\cdot 7^2} ,\b^{2^{2}\cdot5\cdot7}],[\a^{3\cdot 7^2} ,\b^{2^{2}\cdot3\cdot7}],[\a^{3\cdot 7^2} ,\b^{2^{2}\cdot3\cdot5}
			],[\a^{3\cdot 7^2} ,\d^{2^{4}\cdot 5}],[\a^{3\cdot 7^2} ,\d^{2^{4}\cdot 3}],[\a^{3\cdot 7^2} , \g^{2\cdot 5}],[\a^{3\cdot 7^2} ,\g^{2\cdot 3^{2}}],\\ [\a^{3\cdot 7^2} ,\e^5],[\a^{3\cdot 7^2} ,\e^{3^3}],[\a^{2^{2}\cdot 7^2} ,\b^{3\cdot5\cdot7}],[\a^{2^{2}\cdot 7^2} ,\b^{2^{2}\cdot3\cdot7}], [\a^{2^{2}\cdot 7^2} ,\b^{2^{2}\cdot3\cdot5}],[\a^{2^{2}\cdot 7^2} ,\d^{3\cdot 5}],[\a^{2^{2}\cdot 7^2} ,\d^{2^{4}\cdot 3}],\\ [\a^{2^{2}\cdot 7^2} ,\g^{3^{2}\cdot 5}],[\a^{2^{2}\cdot 7^2} ,\g^{2\cdot 3^{2}}],[\a^{2^{2}\cdot 7^2} ,\e^{3^3}],[\a^{2^{2}\cdot 3} ,\b^{3\cdot5\cdot7}],[\a^{2^{2}\cdot 3} ,\b^{2^{2}\cdot5\cdot7}],[\a^{2^{2}\cdot 3} ,\b^{2^{2}\cdot3\cdot7}],[\a^{2^{2}\cdot 3} ,\d^{3\cdot 5}],\\
			[\a^{2^{2}\cdot 3} ,\d^{2^{4}\cdot 5}],[\a^{2^{2}\cdot 3} ,\d^{2^{4}\cdot 3}], [\a^{2^{2}\cdot 3} ,\g^{3^{2}\cdot 5}],[\a^{2^{2}\cdot 3} , \g^{2\cdot 5}],[\a^{2^{2}\cdot 3} ,\g^{2\cdot 3^{2}}],[\a^{2^{2}\cdot3},\e^5],[\a^{2^{2}\cdot3},\e^{3^3}]
		\end{Bmatrix}$
	\end{center}

	Since the commutators $[\a^{3\cdot 7^2} ,\b^{3\cdot5\cdot7}],[\a^{3\cdot 7^2} ,\b^{2^{2}\cdot3\cdot7}],[\a^{3\cdot 7^2} ,\b^{2^{2}\cdot3\cdot5}
	]$ belong to $\gamma_{\om}(G)$ and $gcd(2^{2}\cdot5\cdot7,2^{2}\cdot3\cdot7,2^{2}\cdot3\cdot5)=2^{2}$ we can replace these three elements by the element  $[\a^{3\cdot 7^2} ,\b^{2^{2}}]$. 
	
	We do the same for the rest of the elements of $S_{1}$.\vspace{3mm}\\
	The elements  $[\a^{3\cdot 7^2} ,\d^{2^{4}\cdot 5}],[\a^{3\cdot 7^2} ,\d^{2^{4}\cdot 3}]$ are replaced by $[\a^{3\cdot 7^2} ,\d^{2^{4}}]$.\\
	The elements $[\a^{3\cdot 7^2} , \g^{2\cdot 5}],[\a^{3\cdot 7^2} ,\g^{2\cdot 3^{2}}]$ are replaced by $[\a^{3\cdot 7^2} , \g^{2}]$.\\
	The elements $[\a^{3\cdot 7^2} ,\e^5],[\a^{3\cdot 7^2} ,\e^{3^3}]$ are replaced by $[\a^{3\cdot 7^2} ,\e]$.\\
	The elements $[\a^{2^{2}\cdot 7^2} ,\b^{3\cdot5\cdot7}],[\a^{2^{2}\cdot 7^2} ,\b^{2^{2}\cdot3\cdot7}], [\a^{2^{2}\cdot 7^2} ,\b^{2^{2}\cdot3\cdot5}]$ are replaced by $[\a^{2^{2}\cdot 7^2} ,\b^{3}]$.\\
	The elements $[\a^{2^{2}\cdot 7^2} ,\d^{3\cdot 5}],[\a^{2^{2}\cdot 7^2} ,\d^{2^{4}\cdot 3}]$ are replaced by $[\a^{2^{2}\cdot 7^2} ,\d^{3}]$.\\
	The elements $ [\a^{2^{2}\cdot 7^2} ,\g^{3^{2}\cdot 5}],[\a^{2^{2}\cdot 7^2} ,\g^{2\cdot 3^{2}}]$ are replaced by $ [\a^{2^{2}\cdot 7^2} ,\g^{3^{2}}]$. \\
	The elements $[\a^{2^{2}\cdot 3} ,\b^{3\cdot5\cdot7}],[\a^{2^{2}\cdot 3} ,\b^{2^{2}\cdot5\cdot7}],[\a^{2^{2}\cdot 3} ,\b^{2^{2}\cdot3\cdot7}]$
	are replaced by $[\a^{2^{2}\cdot 3} ,\b^{7}]$.\\
	The elements $[\a^{2^{2}\cdot 3} ,\d^{3\cdot 5}],
	[\a^{2^{2}\cdot 3} ,\d^{2^{4}\cdot 5}],[\a^{2^{2}\cdot 3} ,\d^{2^{4}\cdot 3}]$ are replaced by $[\a^{2^{2}\cdot 3} ,\d]$.\\
	The elements $[\a^{2^{2}\cdot 3} ,\g^{3^{2}\cdot 5}],[\a^{2^{2}\cdot 3} , \g^{2\cdot 5}],[\a^{2^{2}\cdot 3} ,\g^{2\cdot 3^{2}}]$ are replaced by $[\a^{2^{2}\cdot 3} ,\g]$.\\
	The elements $[\a^{2^{2}\cdot3},\e^5],[\a^{2^{2}\cdot3},\e^{3^3}]$ are replaced by $[\a^{2^{2}\cdot3},\e]$.
	\vspace{3mm}

	Now since $[\a^{2^{2}\cdot3},\e],[\a^{3\cdot 7^2} ,\e]\in \gamma_{\om}(G)$ we have that $[\a^{3},\e]\in \gamma_{\om}(G)$ and thus $[\a^{3},\e^{3^3}]\in \gamma_{\om}(G)$. Therefore, since $[\a^{2^{2}\cdot 7},\e^{3^3}]$
	also belongs to $\gamma_{\om}(G)$ and $gcd(3,{2^{2}\cdot 7})=1$ we have that $[\a,\e^{3^3}]\in \gamma_{\om}(G)$. Working in the same way, we replace $[\a^{2^{2}\cdot 3} ,\g]$, $[\a^{3\cdot 7^2} , \g^{2}]$,  $ [\a^{2^{2}\cdot 7^2} ,\g^{3^{2}}]$ by  $ [\a^{2^{2}} ,\g^{3^{2}}]$,  $ [\a^{3} ,\g^{2}]$ and $[\a^{2^{2}\cdot 3} ,\d]$, $[\a^{2^{2}\cdot 7^2} ,\d^{3}]$,  $[\a^{3\cdot 7^2} ,\d^{2^{4}}]$ by  $[\a^{3} ,\d^{2^{4}}]$ and  $[\a^{2^2} ,\d^{3}]$. Therefore the set $S_{1}$ is reduced to the following.
	$$S_{1}=\Big\{	[\a^{3\cdot 7^2} ,\b^{2^{2}}], [\a^{2^{2}\cdot 7^2} ,\b^{3}], [\a^{2^{2}\cdot 3} ,\b^{7}],
	[\a^{2^{2}} ,\g^{3^{2}}],  [\a^{3} ,\g^{2}], 
	[\a^{2^2} ,\d^{3}], [\a^{3} ,\d^{2^{4}}], 
	[\a^{3},\e], [\a,\e^{3^3}]\Big\}$$

	We work similarly for the rest of the elements of $S$. Let $S_{y}$ be the following subset of $S$. 
	\begin{center}
		$\begin{Bmatrix}
			[y_{1},x_{2}],[y_{1},x_{3}],[y_{1},z_{2}],[y_{1},z_{3}],[y_{1},w_{2}],[y_{1},w_{3}], [y_{1},t_{1}], [y_{1},t_{2}], \\
			[y_{2},x_{1}], [y_{2},x_[3]], [y_{2},z_{1}],[y_{2},z_{3}],[y_{2},w_{1}],[y_{2},w_{3}],[y_{2},t_{2}], \\
			[y_{3},x_{1}], 	[y_{3},x_{2}], [y_{3},x_[3]],[y_{3},z_{1}],[y_{3},z_{2}],[y_{3},w_{1}],[y_{3},w_{2}],[y_{3},t_{1}],\\ 	[y_{4},x_{1}], 	[y_{4},x_{2}],
			[y_{4},z_{1}],[y_{4},z_{2}],[y_{4},z_{3}],[y_{4},w_{1}],[y_{4},w_{2}],[y_{4},w_{3}], [y_{4},t_{1}],[y_{4},t_{2}]
		\end{Bmatrix}$
	\end{center}
	The image of $S_{y}$ under $f_{i}^{-1}$, $i\in \{1,2,3,4,5\}$ gives us (under appropriate reduction) the following set $S_{2}$.

	\begin{center}
		\[
		S_{2} = \left.
		\begin{Bmatrix}	[\b^{3\cdot 7} ,\a^{2^{2}}], [\b^{2^{2}\cdot 7} ,\a^{3}], [\b^{2^{2}\cdot 3} ,\a^{7^2}],
			[\b^{3\cdot 5} ,\g^{2}], [\b^{2^{2}\cdot 5} ,\g^{3^2}], \\ [\b^{2^{2}\cdot 3} ,\g^{5}],[\b^{3\cdot 5} ,\d^{2^{4}}], [\b^{2^{2}\cdot 5} ,\d^{3}], [\b^{2^{2}\cdot 3} ,\d^{5}],[\b^{5} ,\e^{3^{3}}], [\b^{3},\e^{5}]\end{Bmatrix}
		\right.
		\]
		
	\end{center}

	If we do the same for the rest of the subsets $S_{z}$, $S_{w}$ and $S_{t}$ and take the union of their images we end up in the following set, whose normal closure in $G$ is $\gamma_{\om}(G)$.

	\begin{center}
		\[
		\gamma_{\omega}(G) = \left.
		{\begin{Bmatrix}[\a^{3\cdot 7^2} ,\b^{2^{2}}], [\a^{2^{2}\cdot 7^2} ,\b^{3}], [\a^{2^{2}\cdot 3} ,\b^{7}],	[\b^{3\cdot 7} ,\a^{2^{2}}], [\b^{2^{2}\cdot 7} ,\a^{3}], [\b^{2^{2}\cdot 3} ,\a^{7^2}], \\ [\a^{2^{2}} ,\g^{3^{2}}],  [\a^{3} ,\g^{2}], 
				[\a^{2^2} ,\d^{3}], [\a^{3} ,\d^{2^{4}}], 
				[\a^{3},\e], [\a,\e^{3^3}],\\
				[\b^{3\cdot 5} ,\g^{2}], [\b^{2^{2}\cdot 5} ,\g^{3^2}],  [\b^{2^{2}\cdot 3} ,\g^{5}], [\g^{3^{2}\cdot 5} ,\b^{2^2}], [\g^{2\cdot 5} ,\b^{3}],  [\g^{2^\cdot 3^2} ,\b^{5}] \\ [\b^{3\cdot 5} ,\d^{2^{4}}], [\b^{2^{2}\cdot 5} ,\d^{3}], [\b^{2^{2}\cdot 3} ,\d^{5}],[\d^{3\cdot 5} ,\b^{2^{2}}], [\d^{2^{4}\cdot 5} ,\b^{3}], [\d^{2^{4}\cdot 3} ,\b^{5}],\\
				[\b^{5} ,\e^{3^{3}}], [\b^{3},\e^{5}],	[\g^{5} ,\e^{3^{3}}], [\g^{3^2},\e^{5}],	[\d^{5} ,\e^{3^{3}}], [\d^{3},\e^{5}]\\
				[\g^{3^{2}\cdot 5} ,\d^{2^{4}}], [\g^{2\cdot 5} ,\d^{3}], [\g^{2\cdot 3^2} ,\d^{5}],[\d^{3\cdot 5} ,\g^{2}], [\d^{2^{4}\cdot 5} ,\g^{3^2}], [\d^{2^{4}\cdot 3} ,\g^{5}]
		\end{Bmatrix}}^{G}
		\right.
		\]
		
	\end{center}
\

\subsection{Calculation of $(N_{3})_{\om}(G)$}

We consider the ten subgroups of $G$ which can be formed as amalgamated free products  defined by distinct vertex groups.
\begin{itemize}
	\item   $H_{\a}^{\b}=\< \a,\b\hspace{1mm}|\hspace{1mm}\a^{42}=\b^{30}\>$
	
		\item   $H_{\a}^{\g}=\< \a,\g \hspace{1mm}|\hspace{1mm}\a^{294}=\g^{30}\>$
			\item   $H_{\a}^{\d}=\< \a,\d \hspace{1mm}|\hspace{1mm}\a^{294}=\d^{120}\>$
		
			\item   $H_{\a}^{\e}=\< \a,\e \hspace{1mm}|\hspace{1mm}\a^{588}=\e^{135}\>$

			\item   $H_{\b}^{\g}=\< \b,\g \hspace{1mm}|\hspace{1mm}\b^{14}=\g^{3}\>$
			\item   $H_{\b}^{\d}=\< \b,\d \hspace{1mm}|\hspace{1mm}\b^{21}=\d^{12}\>$
			\item   $H_{\b}^{\e}=\< \b,\e \hspace{1mm}|\hspace{1mm}\b^{140}=\e^{45}\>$
			\item   $H_{\g}^{\d}=\< \g,\d \hspace{1mm}|\hspace{1mm}\g^{9}=\d^{24}\>$
			\item   $H_{\g}^{\e}=\< \g,\e \hspace{1mm}|\hspace{1mm}\g^{10}=\e^{15}\>$
			\item   $H_{\d}^{\e}=\< \d,\e \hspace{1mm}|\hspace{1mm}\d^{240}=\e^{135}\>$		
\end{itemize}

Now using Theorem \ref{segmNp} for each such subgroup we can calculate each intersection $(N_{3})_{\om}(H_{i}^{j})$, $i\neq j$. Hence we have the following.

\begin{itemize}
	\item   $(N_{3})_{\om}(H_{\a}^{\b})={\a^{3\cdot7}\b^{-3\cdot 5},[\a,\b^3],[\a^{3},\b]}^{G}$
	
	\item   $(N_{3})_{\om}(H_{\a}^{\g})={[\a,\g^{3^2}],[a^{3},\g]}^{G}$
	\item   $(N_{3})_{\om}(H_{\a}^{\d})={\a^{3\cdot 7^2}\d^{-3\cdot 2^{2}\cdot 5},[\a,\d^{3}],[a^{3},\d]}^{G}$
	
	\item   $(N_{3})_{\om}(H_{\a}^{\e})={[\a,\e^{3^{3}}],[\a^{3},\e]}^{G}$

	\item   $(N_{3})_{\om}(H_{\b}^{\g})={[\b,\g]}^{G}$
	\item   $(N_{3})_{\om}(H_{\b}^{\d})= {[\b,\d^3],[\b^{3},\d]}^{G}$
	\item   $(N_{3})_{\om}(H_{\b}^{\e})={\b^{2^{2}\cdot 7}\e^{-3^2},[\b,\e]}^{G}$
	\item   $(N_{3})_{\om}(H_{\g}^{\d})={[\g,\d^3],[\g^{3^2},\d]}^{G}$
	\item   $(N_{3})_{\om}(H_{\g}^{\e})={\g^{2}\e^{-3},[\g,\e]}^{G}$
	\item   $(N_{3})_{\om}(H_{\d}^{\e})={\d^{3\cdot2^4}\e^{-3^3},[\d,\e^{3^3}],[\d^{3},\e]}^{G}$		
\end{itemize}

Now let $S$ be the union $S=\bigcup(N_{3})_{\om}(H_{i}^{j})$, for all $i,j\in \{\a,\b,\g,\d,\e\}$ with $i\neq j$. We will prove that $\tilde{G}=G/S^G$ is residually a finite $3$-group.
First we calculate the abelianization of $\tilde{G}$ which has the following presentation.
$$\tilde{G}^{ab}=\<\a,\b,\g,\d,\e\hspace{1mm}|\hspace{1mm}R_{amag},R_{com}\>$$
 where $R_{amag}$ are the following relations between the generators:
  \begin{center}
 	$\begin{Bmatrix}
 \a^{21}=\b^{15},\a^{294}=\g^{45},\a^{147}=\d^{60},\a^{588}=\e^{135},\b^{14}=\g^{3},\\ \b^{21}=\d^{12}, \b^{28}=\e^{9},\g^{9}=\d^{24}, \g^{2}=\e^{3},\d^{48}=\e^{27}
\end{Bmatrix}$
\end{center}
and $R_{com}$ are the commutators between the generators:
\begin{center}
	$\begin{Bmatrix}
		[\a,\b],[\a,\g],[\a,\d],[\a,\e],[\b,\g],[\b,\d],[\b,\e],[\g,\d],[\g,\e],[\d,\e]
	\end{Bmatrix}$
\end{center}
Now in order to find which abelian group this is, it is sufficient to calculate the Smith normal form of the following matrix.
\begin{center}

$\begin{array}{lcl}
	A & = &  \begin{bmatrix}
	    21 & -15 & 0 & 0 & 0 \\
		294 & 0 & -45 & 0 & 0 \\
	    147 & 0 & 0 & -60 & 0 \\
	    588 & 0 & 0 & 0 & -135 \\
	    0 & 14 & -3 & 0 & 0 \\
	    0 & 21 & 0 & -12 & 0 \\
	    0 & 28 & 0 & 0 & -9 \\
	    0 & 0 & 9 & -24 & 0 \\
	     0 & 0 & 2 & 0 & -3 \\
	    0 & 0 & 0 & 48 & -27 \\    	    
	\end{bmatrix}
\end{array}$
\end{center}

Applying the algorithm we find that

\begin{center}
 $\begin{array}{lcl}
 	SNF(A) & = &  \begin{bmatrix}
 		1 & 0 & 0 & 0 & 0 \\
 		0 & 1 & 0 & 0 & 0 \\
 		0 & 0 & 3 & 0 & 0 \\
 		0 & 0 & 0 & 3 & 0 \\
 		0 & 0 & 0 & 0 & 0 \\
 		0 & 0 & 0 & 0 & 0 \\
 		0 & 0 & 0 & 0 & 0 \\
 		0 & 0 & 0 & 0 & 0 \\
 		0 & 0 & 0 & 0 & 0 \\
 		0 & 0 & 0 & 0 & 0 \\    	    
 	\end{bmatrix}
 \end{array}$

\end{center}
Therefore we conclude that $\tilde{G}^{ab}\cong\Z\times\Z_{3}\times\Z_{3}$ which is residually a finite $3$-group.

Let us now consider the subgroup $N=\<\a^{3},\b^{3},\g^{9},\d^{3},\e^{27}\>$. Notice that this is an abelian  subgroup of $\tilde{G}$. We define the group $\bar{G}=\tilde{G}/N$ which has the following presentation.
$$\bar{G}=\< \a,\b,\g,\d,\e\hspace{1mm}|\hspace{1mm}\a^{3}=\b^{3}=\g^{9}=\d^{3}=\e^{27}=1,\b^{2}=\g^{3}, \b=\e^{9}, \g^{2}=\e^{3},[\b,\g],[\b,\e],[\g,\e]\>$$
or equivalently
$$ \bar{G}=\Z_{3}* \Z_{3}*\< \b,\g,\e\hspace{1mm}|\hspace{1mm}\b^{3}=\g^{9}=\e^{27}=1,\b^{2}=\g^{3}, \b=\e^{9}, \g^{2}=\e^{3},[\b,\g],[\b,\e],[\g,\e]\>$$
Let $H$ be the group $$H=\< \b,\g,\e\hspace{1mm}|\hspace{1mm}\b^{3}=\g^{9}=\e^{27}=1,\b^{2}=\g^{3}, \b=\e^{9}, \g^{2}=\e^{3},[\b,\g],[\b,\e],[\g,\e]\>$$
One can easily see that using Tietze transformations this group is the cyclic group $\Z_{27}$ which is residually a finite $3$-group and thus $\bar{G}$ is residually a finite $3$-group.

Now, it is sufficient to prove that $N$ as a subgroup of the abelianization $\tilde{G}^{ab}$ is isomorphic to $\Z$. Then if  $1\neq g\in \tilde{G}$ and   $\nu_{1}, \nu_{2}$ be the natural epimorphisms from $\tilde{G}$ to $\bar{G}$ and $\tilde{G}^{ab}$ respectively, either $g\notin {N}=Ker\nu_{1}$ and thus $\nu_{1}(g)$ is not trivial or  $g\in {N}$ and thus  $\nu_{2}(g)$ is not trivial. 

Indeed, let $x_{1}=\a^{3}, x_{2}=\b^{3},x_{3}=\g^{9},x_{4}=\d^{3},x_{5}=\e^{27}$. Then $N$ has the following presentation.
$$N=\< x_{1},x_{2},x_{3},x_{3},x_{5}\hspace{1mm}|\hspace{1mm}\bar{R}_{am},\hspace{1mm}[x_{i},x_{j}]=1\text{ for all }i,j\in\{1,\dots,5\}
\>$$
where $\bar{R}_{am}$ is the following set of relations. $$\big\{x_{1}^{7}=x_{2}^{5},x_{1}^{98}=x_{3}^{5},x_{1}^{49}=x_{4}^{20},x_{1}^{196}=x_{5}^{5},x_{2}^{14}=x_{3},x_{2}^{7}=x_{4}^{4},x_{2}^{28}=x_{5},x_{3}=x_{4}^{8},x_{3}^{2}=x_{2}^{5},x_{4}^{16}=x_{5}\big\}$$
Now in order to find which abelian group this is we again calculate the Smith normal form of the following matrix.
\begin{center}

	$\begin{array}{lcl}
		B & = &  \begin{bmatrix}
			7 & -5 & 0 & 0 & 0 \\
			98 & 0 & -5 & 0 & 0 \\
			49 & 0 & 0 & -20 & 0 \\
			196 & 0 & 0 & 0 & -5 \\
			0 & 14 & -1 & 0 & 0 \\
			0 & 7 & 0 & -4 & 0 \\
			0 & 28 & 0 & 0 & -1 \\
			0 & 0 & 1 & -8 & 0 \\
			0 & 0 & 2 & 0 & -1 \\
			0 & 0 & 0 & 16 & -1 \\    	    
		\end{bmatrix}
	\end{array}$
\end{center}

Applying the algorithm we find that

\begin{center}
	$\begin{array}{lcl}
		SNF(B) & = &  \begin{bmatrix}
			1 & 0 & 0 & 0 & 0 \\
			0 & 1 & 0 & 0 & 0 \\
			0 & 0 & 1 & 0 & 0 \\
			0 & 0 & 0 & 1 & 0 \\
			0 & 0 & 0 & 0 & 0 \\
			0 & 0 & 0 & 0 & 0 \\
			0 & 0 & 0 & 0 & 0 \\
			0 & 0 & 0 & 0 & 0 \\
			0 & 0 & 0 & 0 & 0 \\
			0 & 0 & 0 & 0 & 0 \\    	    
		\end{bmatrix}
	\end{array}$

\end{center}
Therefore we conclude that indeed ${N}\cong\Z$ and hence $(N_{3})_{\om}(G)=S^G$.
 \qed


\end{document}